\documentclass[12pt]{article}

\usepackage[english]{babel}
\usepackage{graphicx}
\usepackage{latexsym}
\usepackage{amsfonts}
\usepackage{dsfont}
\usepackage{amscd}
\usepackage{amssymb}
\usepackage{amsmath}
\usepackage{hyperref}
\usepackage{mathrsfs}
\usepackage{amsthm}
\usepackage{tikz}
\usepackage{enumitem}

\allowdisplaybreaks

\setlist[enumerate,1]{label=\textup{(\roman*)}}

\setlist[enumerate]{itemsep=-1pt}
\setlist[itemize]{itemsep=-1pt}
\newlength{\bibitemsep}
\newlength{\bibparskip}
\let\oldthebibliography\thebibliography
\renewcommand\thebibliography[1]{%
	\oldthebibliography{#1}%
	\setlength{\parskip}{\bibitemsep}%
	\setlength{\itemsep}{\bibparskip}%
}

\title{C*-algebras from partial isometric representations of LCM semigroups}

\author{Charles Starling\thanks{Partially supported by an NSERC Discovery Grant and an internal Carleton University research grant. \texttt{cstar@math.carleton.ca}} \and Ilija Tolich\thanks{\texttt{itolich@maths.otago.ac.nz}}}
\usepackage{amssymb,url,hyperref,enumitem}
\newcommand{\NN}{\mathbb{N}}
\newcommand{\CC}{\mathbb{C}}
\newcommand{\ZZ}{\mathbb{Z}}
\newcommand{\Q}{\mathcal{Q}}
\newcommand{\piu}{\pi_{\text{u}}}
\newcommand{\pit}{\pi_{\text{t}}}
\newcommand{\Id}{\text{Id}}

\newcommand{\Irl}{\mathcal{I}_r^l}
\newcommand{\id}{\text{Id}}

\newcommand{\ctspp}{C^{*,r}_{\text{ts}}(P)}
\newcommand{\cpp}{C_{\text{ts}}^*(P)}
\newcommand{\qpp}{\mathcal{Q}_{\text{ts}}(P)}
\newcommand{\ogx}{\mathcal{O}_{(G,X)}}
\newcommand{\sgx}{\mathcal{S}_{(G,X)}}

\newcommand{\I}{\mathcal{I}}
\newcommand{\J}{\mathcal{J}}
\newcommand{\s}{\mathcal{S}}

\newcommand{\B}{\mathcal{B}}

\newcommand{\g}{\mathcal{G}}
\newcommand{\Ef}{\widehat E_0}
\newcommand{\Eu}{\widehat E_\infty}
\newcommand{\Et}{\widehat E_{\text{tight}}}
\newcommand{\Ct}{C^*_{\text{tight}}}

\newcommand{\gt}{\mathcal{G}_{\text{tight}}}
\newcommand{\fs}{\subseteq_{\text{fin}}}
\newcommand{\vp}{\varphi}
\newcommand{\pop}{P^{\text{op}}}
\theoremstyle{plain}
\newtheorem{theo}[subsection]{Theorem}

\newtheorem{lem}[subsection]{Lemma}

\newtheorem{prop}[subsection]{Proposition}

\theoremstyle{definition}
\newtheorem{defn}[subsection]{Definition}

\newtheorem{ex}[subsection]{Example}
\newtheorem{rmk}[subsection]{Remark}

\begin{document}
\maketitle
\begin{abstract}
	We give a new construction of a C*-algebra from a cancellative semigroup $P$ via partial isometric representations, generalising the construction from the second named author's thesis. We then study our construction in detail for the special case when $P$ is an LCM semigroup. In this case we realize our algebras as inverse semigroup algebras and groupoid algebras, and apply our construction to free semigroups and Zappa-Sz\'ep products associated to self-similar groups.
\end{abstract}

MSC Classes: 46L05, 20M18, 18B40


\section{Introduction}
{\bf Background:} C*-algebras associated to semigroups are the subject of an active area of research in operator algebras. If $P$ is a left cancellative semigroup, its reduced C*-algebra is generated by the image of the left regular representation $\lambda: P\to \B(\ell^2(P))$ given by $\lambda_p(\delta_q) = \delta_{pq}$. In his study of Wiener-Hopf operators, Nica \cite{Ni92} defined a suitable universal C*-algebra for semigroups $P$ with a group embedding $P\subseteq G$ which induce a {\em quasi-lattice ordering on $G$}. Li generalized Nica's construction in \cite{Li12} to left-cancellative semigroups which do not necessarily embed in groups. Research on  these algebras and their natural quotients is fruitful and ongoing. In contrast with the group case, picking the left regular representation (rather than the right) affects the construction, and puts left and right multiplication on unequal footing; see the closing remark of \cite{CEL15} and \cite[Remark~7.5]{CaR16} for discussions on choosing the left over the right.

In the algebras above, $P$ is represented by isometries. This paper concerns representing semigroups in C*-algebras by {\em partial isometries}. A representation of a semigroup $P$ in a C*-algebra $A$ is a multiplicative map $\pi:P\to A$, and $\pi$ is called partial isometric if $\pi(p)$ is a partial isometry for all $p\in P$. Multiplicativity of $\pi$ implies that $\pi(p)$ will be a {\em power partial isometry}, i.e. $\pi(p)^n$ is a partial isometry for all $n$. A key example of a power partial isometry is the {\em truncated shift:}
\begin{equation}
J_n: \mathbb{C}^n \to \mathbb{C}^n, \hspace{1cm}
J_n(e_i)= \begin{cases}
e_{i+1} &i<n\\
0& i = n
\end{cases},\label{eq:truncated_shift}
\end{equation}
where $(e_i)_{i\leq n}$ is the standard basis for $\mathbb{C}^n$. Hancock and Raeburn \cite{HR90} considered the operator
\begin{equation}\label{eq:HRJ}
J = \bigoplus_{n=2}^\infty J_n: \bigoplus_{n=2}^\infty \mathbb{C}^n \to \bigoplus_{n=2}^\infty \mathbb{C}^n
\end{equation}
and showed that $C^*(J)$ is the universal C*-algebra generated by a power partial isometry. 

Said another way, $C^*(J)$ is the universal C*-algebra for partial isometric representations of the semigroup $\mathbb{N}$. The second named author's PhD thesis \cite{Tol17} sought to generalize Hancock and Raeburn's work to other semigroups; specifically those which induce  quasi-lattice orders. The pair $(\ZZ, \NN)$ is quasi-lattice ordered with respect to the usual ordering on $\NN$, and \eqref{eq:HRJ} is a direct sum over the {\em principal order ideals} $I_n = \{m\in \NN: m\leq n\}$ with each summand equal to $\ell^2(I_n)$. 

If $P$ is a subsemigroup of a group $G$ and $P\cap P^{-1} = \{1_G\}$, then $P$ induces two partial orders on $G$: $u \leq_l v \iff u^{-1}v\in P$ and $u \leq_r v \iff vu^{-1}\in P$. Note that $\leq_l$ is invariant under left multiplication while $\leq_r$ is invariant under right multiplication. Such semigroups are typically represented by left multiplication operators, so the focus is usually on $\leq_l$. The order $\leq_l$ (or $\leq_r$) is a {\em quasi-lattice order} if every finite set in $G$ which is bounded above has a least upper bound.

A key insight of \cite[1.3.2]{Tol17} is that for partially ordered groups $(G,P)$ which are not necessarily commutative, the map analogous to \eqref{eq:truncated_shift}
\begin{equation}
J^a: P \to \mathcal{B}(\ell^2(I_a)), \hspace{1cm}
J^a_p (\delta_q) = \begin{cases}
\delta_{pq} & pq\in I_a\\
0 & \text{otherwise}
\end{cases},\label{eq:Jpa}
\end{equation}
will not be a representation unless $I_a$ is taken to be an order ideal in the {\em right-invariant} partial order (this distinction is wiped out in commutative cases like $\NN$). So to generalize Hancock and Raeburn's work, \cite{Tol17} starts with $(G,P)$ which is {\em doubly quasi-lattice ordered} (i.e., both $\leq_l$ and $\leq_r$ are quasi-lattice orders) and defines a C*-algebra $C^*_{\text{ts}}(G, P, P^{op})$ generated by direct sums of the operators \eqref{eq:Jpa}, and also defines a suitable universal algebra $C^*(G, P, P^{op})$. It is also shown that the two coincide when $G$ is amenable.

\noindent{\bf Motivation:} Here we show that one can generalize the construction above to general cancellative semigroups $P$. Our motivation is twofold:
\begin{enumerate}
	\item  increase the scope of the construction to include a larger class of semigroups, and
	\item to construct a C*-algebra from semigroup which puts the left and right multiplication structure on equal footing.
\end{enumerate}      
For the first point, we  note that the relations above can be presented on $P$ without mentioning $G$ or the inverse:
\begin{equation}\label{eq:lrorder}
p \leq_l q \iff qP\subseteq pP, \hspace{1cm} p \leq_r q \iff Pq\subseteq Pp,
\end{equation}
and so we can make the same definition \eqref{eq:Jpa} in cases where $P$ does not embed into a group (but note that these relations may no longer be reflexive).

In the usual isometric construction, one particular generalization of Nica's quasi-lattice ordered groups has has received a lot of attention: the {\em right LCM semigroups}.
These are semigroups for which the intersection of any two principal left ideals is either empty or another principal left ideal, and their C*-algebras have been considered by many authors, see \cite{ABLS19, BOS18, BLS16, BLS18, BS16, LL20, LL21, LiB19, NS19, Stam15, Stam17, StLCM}.
Their study is aided by the observation of Norling \cite{No14} that if $P$ is a right LCM semigroup, then $C^*(P)$ can be realized as the universal C*-algebra for a certain enveloping inverse semigroup $I_l(P) \supseteq P$.

Because our construction is incorporating the right multiplication as well, we consider semigroups which satisfy the LCM property for both right ideals and left ideals---we call these {\em LCM semigroups}. Many right LCM semigroups studied in the literature (free semigroups, Baumslag-Solitar monoids, Zappa-Sz\'ep products associated to self-similar groups) happen to also be left LCM. While our construction makes sense for an arbitrary cancellative monoid, all our results are for the LCM case.

For the second point, as we note above, choosing the left regular representation over the right can give different C*-algebras, i.e. $C^*(P)$ is not always isomorphic to $C^*(\pop)$. One of our motivations then is to produce a C*-algebra from a cancellative semigroup which equally expresses the right and left multiplication structure of $P$.

\noindent {\bf Outline}: After giving the general definition of our C*-algebras (which we call $\cpp$ and $\ctspp$), we restrict our attention LCM monoids,  Definition~\ref{def:LCM}.
In this case, we show that one obtains isomorphic algebras from $P$ and $\pop$, Proposition~\ref{prop:iso_opp}. We also crucially show our algebras are generated by an inverse semigroup $\s_P$ containing $P$---this realization is the main source of our results.
It turns out that $\s_P$ is always $E^*$-unitary (Lemma~\ref{lem:E*unitary}).
We show that $\cpp$ is isomorphic to the universal C*-algebra of $\s_P$ (Theorem~\ref{thm:fulliso}) and that $\ctspp$ is isomorphic to the reduced C*-algebra of $\s_P$ (Theorem~\ref{thm:reducediso}).
We then, by definition, take $\qpp$ to be Exel's tight C*-algebra of $\s_P$ (as defined in \cite{Ex08}). Realization of these algebras as inverse semigroup algebras also gives them \'etale groupoid models.

We close the paper by considering some natural examples in Section~\ref{sec:examples}.
The first is that of free monoids. 
When one applies Li's construction to free monoids (and considers their natural boundary quotient) one obtains the Cuntz algebras $\mathcal{O}_n$. 
Our construction yields a very different algebra---the crossed product associated to the full shift (Theorem~\ref{thm:crossedproductfullshift}). 
Our other main example is that of self-similar actions.
We show that our construction results in the same boundary quotient as Li's (Theorem~\ref{thm:SSGboundaryiso}) because in this case tight representations do not see the left ideal structure at all.

\section{Partial isometric representations of semigroups}

\subsection{Preliminaries and notation}
We will use the following general notation. If $X$ is a set and $U\subseteq X$, let Id$_U$ denote the map from $U$ to $U$ which fixes every point, and let $1_{U}$ denote the characteristic function on $U$, i.e. $1_U: X\to \CC$ defined by $1_U(x) = 1$ if $x\in U$ and $1_U(x) = 0$ if $x\notin U$. If $F$ is a finite subset of $X$, we write $F\fs X$.

\subsection{Semigroups and the universal algebra $\cpp$}
A semigroup $P$ is {\em left cancellative} if $pq = pr \implies q = r$ for $p, q, r\in P$, it is {\em right cancellative} if $qp = rp \implies q = r$ for $p, q, r\in P$, and
it is {\em cancellative} if it is both left and right cancellative.
A {\em monoid} is a semigroup with an identity element. 
If $P$ is a monoid, we let $U(P)$ denote the set of invertible elements of $P$.
For $p\in P$, the set $pP = \{pq :  q\in P\}$ is a right ideal, and right ideals of this form are called a {\em principal} right ideals. Similarly, $Pp = \{qp :  q\in P\}$ is a left ideal, and left ideals of this form are called a principal left~ideals. 

As mentioned in the introduction, a map $\pi:P \to A$ from a semigroup $P$ to a C*-algebra $A$ is called a {\em representation} of $P$ if it is multiplicative, it is called {\em (partial) isometric} if $\pi(p)$ is a (partial) isometry for each $p\in P$.

An {\em inverse semigroup} is a semigroup $S$ such that for each $s\in S$ there exists a unique element $s^*$ such that $ss^*s = s$ and $s^*ss^* = s^*$. 
For such a semigroup we let $E(S) = \{e\in S: e^2= e\}$ and call this the set of {\em idempotents}. 
A {\em zero} in $S$ is an element 0 such that $0s = s0 = 0$ for all $s\in S$. 
An inverse semigroup with such a (necessarily unique) element is called an {\em inverse semigroup with zero}, and if $S$ is such we write $S^\times: = S\setminus\{0\}$.
We say that $S$ is {\em E*-unitary} if $s\in S$, $e\in E(S)^\times$ and $se = e$ implies $s\in E(S)$.

The product in an inverse semigroup induces a natural partial order $\leqslant$ on $S$, by saying $s\leqslant t$ if and only if there exists $e\in E(S)$ such that $se =t$.
With this ordering, $E(S)$ is a (meet-) semilattice with meet $e\wedge f = ef$. 

For a set $X$, the {\em symmetric inverse monoid on $X$} is $$\I(X) := \{s:U\to V: U, V\subseteq X, f\text{ is a bijection}\}$$ and is an inverse semigroup when given the operation of composition on the largest possible domain, and when $s^* = s^{-1}$.
Since $st$ must be an element of $S$ for all $s,t\in S$ and it could be that the range of $t$ does not intersect the domain of $s$, $\I(X)$ contains the empty function which we denote 0.
It satisfies $0f = f0 = 0$ for all $f\in \I(X)$, so that $\I(X)$ is an inverse semigroup with zero. 
Here $f\leqslant g$ if and only if $g$ extends $f$ as a function.

For $s\in \I(X)$, let $D_s\subseteq X$ denote its domain. Then it is easy to see that $D_s = D_{s^*s}$ and that $s$ is a bijection from $D_{s^*s}$ to $D_{ss^*}$. Every $s\in \I(X)$ determines a partial isometry $V_s$ in $\B(\ell^2(X))$ with initial space spanned by $\{\delta_x:x\in D_{s^*s}\}$ and final space spanned by $\{\delta_x:x\in D_{ss^*}\}$ determined by $V_s(\delta_x) = \delta_{s(x)}$. It is straightforward to check that the map 
\begin{equation}
\label{eq:Vdef}V:\I(X)\to\B(\ell^2(X)),\hspace{1cm} s\mapsto V_s,
\end{equation}
is an injective inverse monoid homomorphism (i.e. $V_{st} = V_sV_t$ and $V_{s^*} = V_s^*$ for all $s,t\in \I(X)$). 

The assignment $Y\mapsto \Id_Y$ is a monoid isomorphism from $\mathcal{P}(X)$ (the power set of $X$, which is a monoid when given set intersection) to $E(\I(X))$. For $s\in \I(X)$ and $Y\subseteq X$, the elements $s\Id_Ys^*$ and $s^*\Id_Ys$ are both idempotents, and hence the identity functions on subsets of $X$. It is straighforward to check that if we set
\begin{equation}\label{eq:Y|s}
\left.Y\right|_s :=s^*(Y)\cap D_{ss^*} , 
\end{equation}
then we have $s\Id_Ys^* = \Id_{\left.Y\right|_{s^*}}$ and $s^*\Id_Ys = \Id_{\left.Y\right|_s}$. For $Y\subseteq X$ let $e_Y\in \B(\ell^2(X))$ be the orthogonal projection onto $\ell^2(Y)$. Then because $V$ is a homomorphism we have
\begin{equation}\label{eq:VeV}
V_se_YV_s^* = e_{\left.Y\right|_{s^*}}, \hspace{1cm}V_s^*e_YV_s = e_{\left.Y\right|_{s}}.
\end{equation}

Let $P$ be a left cancellative semigroup. 
For $a\in P$ we consider the elements of $P$ which are $\leq_r$ below $a$ (see \eqref{eq:lrorder}):
\[
I_a := \{x\in P: Pa\subseteq Px\} = \{x\in P: x\leq_r a\}.
\]
Note that $xy\in I_a$ implies $y\in I_a$ (because then $Pa\subseteq Pxy\subseteq Py)$.
Left cancellativity implies that left multiplication by $p$ induces a bijection from $\{x:px\in I_a\}$ to $\{px:px\in I_a\}$. Thus if we define
\[
J^a_p\delta_x = \begin{cases}
\delta_{px} &\text{if }px\in I_a\\
0&\text{otherwise}
\end{cases},
\]
we have that $J_p^a$ is a partial isometry in $\B(\ell^2(I_a))$.
\begin{lem}\label{lem:Ja_partial_isometric_reps}
	Let $P$ be a left cancellative semigroup. Then for each $a\in P$, the map $p\mapsto J^a_p$ is a partial isometric representation of $P$.	
\end{lem}
\begin{proof}
	Suppose that we have $p,q,a,x\in P$ with $x\in I_a$. If $pqx\in I_a$ then as we noted above we also have $qx\in I_a$, and thus $J^a_pJ^a_q\delta_x = J^a_p\delta{qx} = \delta_{pqx} = J^a_{pq}\delta_x$. If $pqx\notin I_a$, then both $J^a_pJ^a_q\delta_x$ and $J^a_{pq}\delta_x$ are zero. In both cases, $J_pJ_q = J^a_{pq}$. 
\end{proof}
We now consider the direct sum of these representations
\[
J: P\to \B\left(\bigoplus_{a\in P}\ell^2(I_a)\right), \hspace{1cm}J_p:= \bigoplus_{a\in P} J_p^a.
\]
Let 
\begin{equation}\label{eq:generalDelta}\Delta = \{(a, x)\in P\times P :  x\in I_a\}.
\end{equation}
We naturally identify $\bigoplus_{a\in P}\ell^2(I_a)$ with $\ell^2(\Delta)$ via $\ell^2(I_a) \ni \delta_{x}^a \mapsto \delta_{(a,x)}\in \ell^2(\Delta)$.
We will then write the standard orthonormal basis of $\ell^2(\Delta)$ as $\{\delta_{x}^{a}: x\in I_a\}$, and using this identification we have
\begin{equation}\label{eq:Jpdef}
J_p(\delta_x^{a}) = \begin{cases}
\delta_{px}^{a} & \text{if }px\in I_a
\\0&\text{otherwise}
\end{cases}, \hspace{1cm}
J_p^*(\delta_x^{a}) = \begin{cases}
\delta_{p_1}^{a} & \text{if }x= pp_1\\0&\text{otherwise}
\end{cases}.
\end{equation}
\begin{defn}\label{def:CtsDef}
	Let $P$ be a left cancellative semigroup and let $J$ be as above. 
	We let $\ctspp$ denote the C*-algebra generated by the set $\{J_p: p\in P\}\subseteq \B\left(\ell^2(\Delta)\right)$, and call this the {\em reduced truncated shift C*-algebra of $P$}.
\end{defn}

Similar to \cite[Definition~2.15]{Tol17} both the name and the subscript ``ts'' are meant to indicate that it is generated by generalized truncated shift operators, as described in \cite[Lemma~2.12]{Tol17}.

To make Definition~\ref{def:CtsDef} we have only needed to assume left cancellativity of $P$. Our motivation (discussed in the introduction) requires $P$ to be right cancellative as well. In addition, the following description of the universal algebra is clearest when $P$ is assumed to have an identity. Therefore, {\bf for the rest of the paper we will consider only the case where $P$ is a cancellative monoid}.

Fix a cancellative monoid $P$ now, with identity $1$. Take $\Delta$ as in \eqref{eq:generalDelta} and note that in this case we have 
\begin{equation*}\label{eq:DeltaDef}
\Delta = \{(bx, x)\in P\times P :  b,x\in P\}.
\end{equation*}
Then taking $X := \Delta$ in \eqref{eq:Vdef}, we have that $J_p$ is in the image of $V$ for all $p\in P$, so that we can define $v_p:= V^{-1}(J_p)$. Note that because $J_{pq} = J_pJ_q$ for all $p,q\in P$, we have
\begin{equation}\label{eq:vpq=vpvq}
v_{pq} = V^{-1}(J_{pq}) = V^{-1}(J_pJ_q) = V^{-1}(J_p)V^{-1}(J_q) = v_pv_q.
\end{equation}
From \eqref{eq:Jpdef} we see that
\begin{equation}
\label{eq:Dvp}D_{v_p^*v_p} = \{(bpx, x): b,x\in P\},\hspace{1cm} D_{v_pv_p^*} = \{(bpx, px): b,x\in P\},
\end{equation}  
and $v_p(bpx,x) = (bpx, px)$. 
\begin{lem}
	Let $Y\subseteq \Delta$ and set $Y_p:= \left.Y\right|_{v_p}$ and $Y^p:= \left.Y\right|_{v_p^*}$ (see \eqref{eq:Y|s}). Then 
	\begin{equation}\label{eq:Ypp}
	Y_p = \{(bpx, px) :  (bpx, x)\in Y\},\hspace{1cm} Y^p = \{(bpx, x) :  (bpx, px)\in Y \}.
	\end{equation}
\end{lem}
\begin{proof}
	We prove the statement for $Y_p$, $Y^p$ is similar. For $b,x\in P$, we have $\gamma := (bpx,px)$ and $(bpx, x)\in Y$ if and only if $\gamma = v_p^*(bpx,x)\in v_p^*(Y)$ and $\gamma\in D_{v_pv_p^*}$ by \eqref{eq:Dvp}. 
\end{proof}
\begin{defn}\label{def:constructible}
	Let $P$ be a cancellative monoid. 
	Then the {\em set of constructible subsets} of $\Delta$, denoted $\J(P)$,  is the smallest collection of subsets of $\Delta$ which is closed under finite intersections,
	contains $Y_p$ and $Y^p$ whenever $Y\in \J(P)$ and $p\in P$, and
	contains $\Delta$ and $\emptyset$.
\end{defn}
Using the notation \eqref{eq:Ypp}, for each $p\in P$ we have
\begin{equation}\label{eq:littlevpdef}
v_p: \Delta^p\to \Delta_p, \hspace{1cm}
v_p(bpx,x) = (bpx,px),
\end{equation}
and putting \eqref{eq:VeV} and \eqref{eq:Ypp} together we have
\begin{align}
v_p\Id_Yv_p^* &= \Id_{Y_p},&v_p^*\Id_Yv_p &= \Id_{Y^p},&\text{for all }Y\subseteq \Delta, p\in P,\label{eq:vIdv}\\
J_pe_YJ_p^* &= e_{Y_p},&J_p^*e_YJ_p &= e_{Y^p},&\text{for all }Y\subseteq \Delta, p\in P,\label{eq:JeJ}
\end{align}
and so $e_Y\in \ctspp$ for all $Y\in \J(P)$. 

Our goal is to define a universal algebra for partial isometric representations of $P$, but in analogy with \cite{Li12} we want our definition to include a set of projections isomorphic to $\J(P)$ satisfying the relations \eqref{eq:JeJ}. Before giving our definition, we will give another description of $\J(P)$ as the idempotent semilattice of the inverse monoid generated by the $v_p$. 
%
%
%
Define
\begin{equation}\label{eq:s}
\Irl(P) = (\text{the inverse semigroup generated by $\{v_p :  p\in P\}$ inside $\I(\Delta)$})\cup\{0\}.
\end{equation}
\begin{lem}
	$E(\Irl(P)) = \{\id_Y : Y\in \J(P)\}$. Hence, $E(\Irl(P))$ and $\J(P)$ are isomorphic as semilattices.
\end{lem}
\begin{proof}
	By \eqref{eq:vpq=vpvq}, we can write a general nonzero element $s\in \Irl(P)$ in the form
	\[
	s = v_{p_1}v_{q_1}^*v_{p_1}v_{q_1}^*\cdots v_{p_n}v_{q_n}^*
	\]
	for some $p_1,\dots p_n, q_1, \dots q_n \in P$. So we calculate
	\begin{align*}
	ss^* &= v_{p_1}v_{q_1}^*v_{p_1}v_{q_1}^*\cdots v_{p_n}v_{q_n}^*v_{q_n}v_{p_n}^*\cdots v_{q_1}v_{p_1}^*\\
	&= v_{p_1}v_{q_1}^*v_{p_1}v_{q_1}^*\cdots v_{p_n}\id_{\Delta^{q_n}}v_{p_n}^*\cdots v_{q_1}v_{p_1}^*\\
	&= v_{p_1}v_{q_1}^*v_{p_1}v_{q_1}^*\cdots v_{q^*_{n-1}}\id_{(\Delta^{q_n})_{p_n}}v_{q_{n-1}}\cdots v_{q_1}v_{p_1}^*\\
	&\vdots\\
	& = \id_{(\cdots(\Delta^{q_n})_{p_n})^{q_{n-1}})_{p_{n-1}})\cdots)_{p_1}}.
	\end{align*}
	Hence $ss^*$ is of the form $\id_Y$ for some $Y\in \J(P)$, and since $E(\Irl(P))$ coincides with the set of all such elements (together with zero), we have $E(\Irl(P)) \subseteq\{\id_Y : Y\in \J(P)\}$.
	
	Now, let $B = \{Y\subseteq \Delta:\id_Y\in E(\Irl(P))\}$. Then $B$ satisfies all of the conditions of Definition~\ref{def:constructible}, and since $\J(P)$ is the smallest such set we have $\J(P)\subseteq B$, and so $\{\id_Y : Y\in \J(P)\}\subseteq E(\Irl(P))$.
\end{proof}

\begin{defn}\label{def:cstarmonoid}
	Let $P$ be a cancellative monoid. Then we define $\cpp$ to be the universal unital C*-algebra generated by a set of partial isometries $\{S_p :  p\in P\}$ and projections $\{f_Y :  Y\in \J(P)\}$ such that
	\begin{enumerate}
		\item $S_pS_q = S_{pq}$ for all $p,q\in P$,\label{it1:universal_def}
		\item $f_Yf_Z = f_{Y\cap Z}$ for all $Y, Z\in \J(P)$,\label{it2:universal_def}
		\item $f_\Delta = 1, f_\emptyset = 0$,\label{it3:universal_def}
		\item $S_pf_YS_p^* = f_{Y_p}$ for all $Y\in \J(P)$, $p\in P$, and\label{it4:universal_def}
		\item $S_p^*f_YS_p = f_{Y^p}$ for all $Y\in \J(P)$, $p\in P$.\label{it5:universal_def}
	\end{enumerate}
\end{defn}
We claim that the sets $\{e_Y:Y\in \J(P)\}$ and $\{J_p:p\in P\}$ satisfy the relations in Definition~\ref{def:cstarmonoid}. First, Lemma~\ref{lem:Ja_partial_isometric_reps} implies \ref{it1:universal_def} is satisfied. Items \ref{it2:universal_def} and \ref{it3:universal_def} are clearly satisfied. Furthermore, \eqref{eq:JeJ} implies that $e_Y\in \ctspp$ for all $Y\in \J(P)$ and also that items \ref{it4:universal_def} and \ref{it5:universal_def} are satisfied. Since the generators all have norm one and the relations are satisfied in the particular C*-algebra $\ctspp$, the universal algebra defined in Definition~\ref{def:cstarmonoid} exists by \cite[Section~1]{Bl85}.

In what follows we study this C*-algebra for LCM monoids.
\subsection{LCM monoids}

The works \cite{ABLS19, BOS18, BLS16, BLS18, BS16, LL20, LL21, LiB19, NS19, Stam15, Stam17, StLCM} and others focus on a special class of left cancellative semigroups, called the {\em right LCM semigroups}. 
Here we define a natural corresponding notion in our setting.


\begin{defn}\label{def:LCM}
	A monoid $P$ is called {\em right LCM} if for all $p,q\in P$, $pP\cap qP$ is either empty or equal to $rP$ for some $r\in P$. It is called {\em left LCM} if for all $p,q\in P$, $Pp\cap Pq$ is either empty or equal to $Pr$ for some $r\in P$. We say $P$ is an {\em LCM monoid} if it is both right LCM and left LCM. 
\end{defn}

\begin{ex}
	{\bf Doubly quasi-lattice ordered groups}
	
	These are the prototype for our definition, and were defined in \cite{Tol17}. These are a special class of Nica's quasi-lattice ordered groups \cite{Ni92}.
	
	Let $G$ be a group and suppose $P\subseteq G$ is a subsemigroup of $G$ such that $P\cap P^{-1} = \{1_G\}$.
	One defines two partial orders on $G$ as follows:
	\begin{enumerate}
		\item  $u\leq_l v$ $\iff$ $u^{-1}v\in P \iff v\in uP\iff vP\subseteq uP$.
		\item $u\leq_r v \iff vu^{-1}\in P \iff v\in Pu\iff Pv\subseteq Pu$.
	\end{enumerate} 
	Then $(G,P)$ is said to be a {\em doubly quasi-lattice ordered group} (see \cite[Definition~2.2]{Tol17}) if both of the following are satisfied:
	\begin{enumerate}
		\item Every finite set with a common upper bound for $\leq_l$ has a least upper bound for $\leq_l$.
		\item Every finite set with a common upper bound for $\leq_r$ has a least upper bound for $\leq_r$.
	\end{enumerate}
	
	Given such a pair $(G,P)$, $P$ is an LCM monoid.
	To see this, first notice that $P$ must be cancellative by virtue of being contained in a group, and that $P\cap P^{-1} = \{1_G\}$ means that $P$ is a monoid.
	The two conditions in the definition applied to the finite set $\{p,q\}$ for $p,q\in P$ imply that $pP\cap qP$ is either empty or equal to $rP$, where $r$ is the least upper bound of $p$ and $q$ with respect to $\leq_l$. 
	Likewise, $Pp\cap Pq$ is either empty or equal to $Ps$ where $s$ is the least upper bound with respect to $\leq_r$. Hence, $P$ is an LCM monoid.
	
	Notice in this case that the elements $r$ and $s$ are {\em unique}. 
	This is not necessarily true for general LCM monoids, as $rP= ruP$ and $Ps = Pus$ for any invertible element $u$.
\end{ex}

\begin{ex}\label{ex:FreeSemigroups}
	{\bf Free semigroups}
	
	Let $X$ be a finite set (or {\em alphabet}). For $n\in\NN$ we write an element $(a_1, a_2, \dots, a_n)\in X^n$ in the condensed way $a_1a_2\cdots a_n$, and call these elements {\em words} of {\em length} $n$. For $\alpha\in X^n$ we write $|\alpha| = n$. Define $X^0 = \{\epsilon\}$, call $\epsilon$ the {\em empty word}, and let 
	\[
	X^* = \bigcup_{n\geq 0}X^n.
	\]
	Then $X^*$ becomes a monoid when given the operation of concatenation: if $\alpha,\beta\in X^*$ their product is
	\[
	\alpha\beta = \alpha_1\alpha_2\cdots\alpha_{|\alpha|}\beta_1\beta_2\cdots\beta_{|\beta|}.
	\]
	If $w = \alpha\beta$, we say that $\alpha$ is a {\em prefix} of $w$ and that $\beta$ is a {\em suffix} of $w$.
	We also say that $w$ {\em starts with} $\alpha$ and {\em ends with} $\beta$. 
	We will say that $\alpha$ and $\beta$ {\em agree} if either $\alpha$ is a prefix of $\beta$ or $\beta$ is a prefix of $\alpha$.

	This semigroup is clearly cancellative. For $\alpha\in X^*$, $\alpha X^*$ is the set of words which begin with $\alpha$, and $\alpha X^* \cap \beta X^*$ is empty unless $\alpha$ is a prefix of $\beta$ (in which case $\alpha X^* \cap \beta X^* = \beta X^*$) or $\beta$ is a prefix of $\alpha$ (in which case $\alpha X^* \cap \beta X^* = \alpha X^*$). 
	Hence, $X^*$ is right LCM. 
	
	Similarly, $X^*\alpha$ is the set of words which end with $\alpha$, and $X^*\alpha \cap X^*\beta$ is empty unless $\alpha$ is a suffix of $\beta$ (in which case $X^*\alpha \cap X^*\beta =  X^*\beta$) or $\beta$ is a suffix of $\alpha$ (in which case $X^*\alpha \cap X^*\beta = X^*\alpha$). 
	Thus $X^*$ is left LCM and hence an LCM monoid.
\end{ex}

\begin{ex}\label{ex:ssg}
	{\bf Self-similar actions}
	
	We now describe an example which is not a doubly quasi-lattice ordered group.
	Let $X^*$ be as in Example~\ref{ex:FreeSemigroups}, and let $G$ be a group.
	Suppose that $G$ acts on $X^*$ on the left faithfully by length-preserving bijections, i.e.,
	\[
	G\times X^* \to X^*,\hspace{1cm}
	(g, \alpha)\mapsto g\cdot \alpha,\hspace{1cm}
	g\cdot X^n = X^n\text{ for all }g\in G, n\geq 0.
	\]
	Suppose also that we have a {\em restriction map}
	\[
	G\times X^* \to G, \hspace{1cm}
	(g, \alpha)\mapsto \left.g\right|_\alpha,
	\]
	which satisfies
	\[
	g\cdot(\alpha\beta) = (g\cdot \alpha)(\left. g\right|_\alpha\cdot \beta)
	\]
	for all $\alpha,\beta\in X^*$ and for all $g\in G$. 
	Then we call the pair $(G, X)$ a {\em self-similar action}. 
	We record two properties which a self-similar action might satisfy.
	\begin{defn}
		Let $(G,X)$ be a self-similar action. 
		\begin{enumerate}
			\item \cite[Definition 5.4]{EP17} $(G, X)$ is called {\em pseudo-free} if $g\cdot \alpha = \alpha$ and $\left.g\right|_\alpha = 1_G$ for some $\alpha\in X^*$ implies that $g = 1_G$.
			\item \cite[p.13]{Nek04} $(G, X)$ is called {\em recurrent} if for any $h\in G$ and for any $\alpha,\beta\in X^*$ with $|\alpha| = |\beta|$, there exists $g\in G$ such that 
			\[
			g\cdot \alpha = \beta \hspace{1cm}\text{and}\hspace{1cm} \left.g\right|_\alpha = h.
			\]
		\end{enumerate}
	\end{defn}
	To any self-similar action one can associate a right LCM semigroup.
	The {\em Zappa-Sz\'ep product} $X^*\bowtie G$ is the set $X^*\times G$ with the operation
	\[
	(\alpha, g)(\beta, h) = (\alpha(g\cdot \beta), \left.g\right|_\beta h).
	\]
	It was shown in \cite[Theorem~3.8]{BRRW14} that $X^*\bowtie G$ is always a right LCM semigroup. 
	It is well-known that $X^*\bowtie G$ is right cancellative if and only if $(X, G)$ is pseudo-free, see \cite[Proposition~3.11]{LW15} or \cite[Lemma~3.2]{ES16} for proofs.
	
	We have the following characterization for $X^*\bowtie G$  to be LCM.
	\begin{lem}
		\label{lem:SSGLCM}
		Let $(G, X)$ be a self-similar action. Then for any principal left ideals their intersection $X^*\bowtie G(\alpha, g)\cap X^*\bowtie G(\beta, h)$ is either empty or equal to $X^*\bowtie G(\alpha, g)$ or $X^*\bowtie G(\beta, h)$. In particular, $X^*\bowtie G$ is a left LCM monoid if and only if $(G,X)$ is pseudo-free.
	\end{lem}
	\begin{proof}
		Let $(\alpha, g),(\beta, h)\in X^*\bowtie G$ and suppose that  $X^*\bowtie G(\alpha, g)\cap X^*\bowtie G(\beta, h)\neq \emptyset$.
		We suppose, without loss of generality, that $|\alpha|\geq |\beta|$. We claim that $$X^*\bowtie G(\alpha, g)\cap X^*\bowtie G(\beta, h)=X^*\bowtie G(\alpha, g).$$
		
		Since $X^*\bowtie G(\alpha, g)\cap X^*\bowtie G(\beta, h)\neq \emptyset$ we must have some $(\gamma,j),(\lambda,k)\in X^*\bowtie G$ such that
		\begin{align*}
		(\gamma,j)(\alpha, g)&=(\lambda,k)(\beta, h),\\
		(\gamma(j\cdot \alpha),j|_\alpha g)&=(\lambda(k\cdot \beta),k|_\beta h).
		\end{align*}
		This implies  $\gamma(j\cdot \alpha)= \lambda(k\cdot \beta)$ and $j|_\alpha g=k|_\beta h$. This indicates that $X^*(j\cdot \alpha)\cap X^*(k\cdot \beta)\neq \emptyset$. Since the action is length preserving $|j\cdot \alpha|\geq|k\cdot \beta|$, therefore $X^*(j\cdot \alpha)\cap X^*(k\cdot \beta)=X^*(j\cdot \alpha)$ from the properties of the free monoid. Thus there exists some $\theta\in X^*$ such that $(j\cdot \alpha)=\theta (k\cdot \beta)$.
		
		We now can prove our claim by showing that $(\alpha, g)\in X^*\bowtie G(\beta, h)$ and hence $X^*\bowtie G(\alpha, g)\subseteq X^*\bowtie G(\beta, h)$.
		
		We will show that $(\alpha, g)=((j^{-1}\cdot \theta),j^{-1}|_\theta k)(\beta, h)$. Compute, using the Zappa-Sz\'ep properties in \cite[Lemma 3.1]{BRRW14}:
		\begin{align*}
		((j^{-1}\cdot \theta),j^{-1}|_\theta k)(\beta, h)&=((j^{-1}\cdot \theta)((j^{-1}|_\theta k)\cdot\beta),(j^{-1}|_\theta k)|_{\beta} h).
		\end{align*}
		To make this easier to follow we handle the two components separately.
		\begin{align*}
		(j^{-1}\cdot \theta)((j^{-1}|_\theta k)\cdot\beta)&=(j^{-1}\cdot \theta)((j^{-1}|_\theta\cdot (k\cdot\beta)) &\text{by (B2)}\\
		&=j^{-1}\cdot(\theta(k\cdot \beta))&\text{(B5)}\\
		&=j^{-1}\cdot(j\cdot \alpha)& \text{because $\theta(k\cdot \beta)=j\cdot\alpha$}\\
		&=(j^{-1}j)\cdot\alpha\\
		&=\alpha.
		\end{align*}
		\begin{align*}
		(j^{-1}|_\theta k)|_{\beta} h&=(j^{-1}|_\theta)|_{k\cdot \beta}k|_\beta h&\text{(B8)}\\
		&=j^{-1}|_{\theta (k\cdot \beta)}k|_\beta h&\text{(B6)}\\
		&=j^{-1}|_{j\cdot\alpha}k|_\beta h& \text{because $\theta(k\cdot \beta)=j\cdot\alpha$)}\\
		&=j^{-1}|_{j\cdot\alpha}j|_\alpha g &\text{$k|_\beta h=j|_\alpha g$ by assumption}\\
		&=(j^{-1}j)|_\alpha g&\text{(B8)}\\
		&=e|_\alpha g\\
		&=g.
		\end{align*}
		Therefore $((j^{-1}\cdot \theta),j^{-1}|_\theta k)(\beta, h)=(\alpha, g)$. We have thus proved our claim and shown that $X^*\bowtie G(\alpha, g)\cap X^*\bowtie G(\beta, h)=X^*\bowtie G(\alpha, g).$
		
		So in any case we have that $X^*\bowtie G$ satisfies the LCM property on both the right and left, and is always left cancellative. Therefore it is an LCM monoid if and only if it is right cancellative. By \cite[Proposition~3.11]{LW15}, this is equivalent to $(G,X)$ being pseudo-free.
	\end{proof}	
	In the case that $(G, X)$ is recurrent, the principal left ideals are linearly ordered by inclusion.
	\begin{lem}\label{lem:recurrent}
		Let $(G, X)$ be a self-similar action. If $(G,X)$ is recurrent, then the set of principal left ideals of $X^*\bowtie G$ is given by
		\[
		\{I_n :  n\in\ZZ, n\geq 0\},\hspace{1cm} \text{where}\hspace{1cm}			
		I_n = \{ (\beta, h) :  |\beta|\geq n\}.
		\]
		In particular, the set of principal left ideals of $X^*\bowtie G$ is linearly ordered by inclusion.		
	\end{lem}
	\begin{proof}
		Take $(\alpha, g)\in X^*\bowtie G$. 
		We claim that $X^*\bowtie G(\alpha, g) = I_{|\alpha|}$. 
		That  $X^*\bowtie G(\alpha, g) \subseteq I_{|\alpha|}$ is clear, because multiplying elements of $X^*\bowtie G$ increases the length of the first coordinate.
		So suppose that $(\beta, h)\in I_n$, and write $\beta = \gamma\delta$ with $|\delta| = \alpha$. 
		Find $k\in G$ such that $k\cdot \alpha = \delta$ and $\left.k\right|_\alpha = hg^{-1}$. 
		Then
		\[
		(\gamma, k)(\alpha, g) =(\gamma (k\cdot \alpha), \left.k\right|_\alpha g)                = (\gamma \delta, hg^{-1}g)
		= (\beta, h).	
		\]
		Hence $(\beta, h)\in X^*\bowtie G(\alpha, g)$, proving that $X^*\bowtie G(\alpha, g) = I_n$. 
		
		To complete the proof, we simply notice that $I_n \cap I_m = I_{\min\{m,n\}}$, so the intersection of two principal left ideals is another principal left ideal. 
	\end{proof}

	We will not go into further detail on self-similar actions here---the interested reader is directed to \cite{Nek05}, \cite{Nek09}, \cite{LRRW14}, \cite{BRRW14}, or \cite{ES16}.
	
	A natural question to ask about a given cancellative semigroup is: does it embed into a group?
	Lawson and Wallis proved in \cite[Theorem~5.5]{LW15} that $X^*\bowtie G$ embeds into a group if and only if it is cancellative, and this occurs if and only if $(G,X)$ is pseudo-free.
	Hence, all of our examples above are group-embeddable.
	
	We are thankful to an anonymous referee for pointing out that not every LCM semigroup embeds into a group: in their paper about interval monoids arising from posets, Dehornoy and Wehrung \cite[Proposition~B]{DW17} constructed an LCM monoid which does not.
\end{ex}
We now prove some general facts about LCM monoids and the transformations \eqref{eq:littlevpdef}.
\begin{lem}\label{lem:constructibleLCM}
	Let $P$ be an LCM monoid, and let $p,q,r\in P$. 
	Then
	\begin{enumerate}
		\item \label{it1:constructibleLCM}$\Delta_p = \Delta_q \iff pP = qP \iff p = qu$ for some $u\in U(P)$,
		\item \label{it2:constructibleLCM}$\Delta^p = \Delta^q \iff Pp = Pq \iff p = uq$ for some $u\in U(P)$,
		\item \label{it3:constructibleLCM}$\Delta_p\cap \Delta_q =\begin{cases} \Delta_r &\text{if }pP\cap qP = rP\\\emptyset&\text{if }pP\cap qP = \emptyset\end{cases},$
		\item \label{it4:constructibleLCM} $\Delta^p\cap \Delta^q =\begin{cases} \Delta^r &\text{if }Pp\cap Pq = Pr\\\emptyset&\text{if }Pp\cap Pq = \emptyset\end{cases},$
		\item \label{it5:constructibleLCM}$(\Delta_p\cap \Delta^q)_r =\begin{cases}  \Delta_{rp}\cap \Delta^{r_1} &\text{if }Pr\cap Pq = Pk\text{ with }r_1r=q_1q = k\\\emptyset&\text{if }Pr\cap Pq = \emptyset\end{cases},$
		\item \label{it6:constructibleLCM}$(\Delta_p\cap \Delta^q)^r =\begin{cases} \Delta_{r_1}\cap \Delta^{qr}&\text{if }pP\cap rP = kP\text{ with }pp_1= rr_1 = k\\\emptyset&\text{if }pP\cap rP = \emptyset\end{cases}.$	
	\end{enumerate}
	Hence, the set of constructible ideals $\J(P)$ has the closed form
	\begin{equation}
	\J(P) = \{\Delta_p\cap \Delta^q: p,q\in P\}\cup\{\emptyset\},\label{eq:JPformLCM}
	\end{equation}
\end{lem}
\begin{proof}
	\ref{it1:constructibleLCM} 
	First suppose that $\Delta_p = \Delta_q$.
	Thus for every $(bpx,px)\in \Delta_p$, there exists $(aqy,y)\in \Delta$ such that $(aqy,qy) = (bpx,px)$.
	Since $px\in qP$ for all $x$, we have $pP\subseteq qP$. 
	By a symmetric argument we get $qP\subseteq pP$, so we have $pP = qP$.  
	
	If $pP = qP$ then $q\in pP$ and $p\in qP$ implies $p = qu$ and $q = pv$ for some $u,v\in P$. 
	Hence $p = pvu$, and cancellativity implies $vu = 1$, so $u$ is invertible. 
	
	Finally if $p = qu$ for some $u\in U(P)$ then $\Delta_p\ni (bpx,x) = (bqux,x)\in \Delta_q$. But since $q = pu^{-1}$ a symmetric argument gives $\Delta_q\subseteq \Delta_p$. 
	
	\ref{it2:constructibleLCM} Similar to \ref{it1:constructibleLCM}.
	
	\ref{it3:constructibleLCM} First, suppose that $pP\cap qP = rP$, and therefore we can find $p_1, q_1\in P$ such that $pp_1 = qq_1 = r$. 
	The intersection $\Delta_p\cap \Delta_q= \{(bpx, px) :  b, x\in P \}\cap \{(aqy, qy)  :  a,y\in P\}$
	is nonempty, because the element $(pp_1, pp_1) = (qq_1, qq_1) = (r, r)$ is common to both (taking $a=b=1$, $x=p_1$ and $y=q_1$).
	We claim that $\Delta_p\cap \Delta_q = \Delta_r$. 
	Suppose $(bpx,px) = (aqy, qy) \in \Delta_p\cap \Delta_q$.
	Then since $px = qy$, this element is in $pP\cap qP = rP$, so there exists $c\in P$ such that $px= qy = rc$.
	Hence $(bpx, px) = (brc,rc)\in \Delta_r$.
	On the other hand, if $(brc, rc)\in \Delta_r$, then $(brc,rc )= (bpp_1c, pp_1c) = (bqq_1c, qq_1c)$ is clearly in $\Delta_p\cap \Delta_q$. 
	Hence, $\Delta_p\cap \Delta_q = \Delta_r$.
	
	If $pP\cap qP = \emptyset$, then the above shows that $\Delta_p\cap \Delta_q = \emptyset$, and hence the first product is zero.
	
	\ref{it4:constructibleLCM} Similar to \ref{it3:constructibleLCM}.
	
	\ref{it5:constructibleLCM} If $\gamma\in \Delta_p\cap \Delta^q$, then $\gamma = (bpx,px) = (cqy,y)$ for some $b,c,x,y\in P$. 
	This implies $y = px$ and hence $b = cq$. Thus
	\begin{equation}\Delta_p\cap \Delta^q = \{(cqpx,px): c,x\in P\}\label{eq:DeltapDeltaq}
	\end{equation}
	which implies that
	\begin{align*}
	(\Delta_p\cap \Delta^q)_r &=\{(arz,rz): (arz,z)=(cqpx,px) \text{ for some }a,c,z,x\in P \}\nonumber \\
	&= \{(arpx, rpx): ar = cq\text{ and }  a,c, x\in P\}\nonumber
	\end{align*}
	If $Pr\cap Pq = \emptyset$ then no such $a,c\in P$ can exist, so $(\Delta_p\cap \Delta^q)_r$ is empty.
	Otherwise, take $\gamma = (arpx,rpx) = (cqpx,rpx)\in (\Delta_p\cap \Delta^q)_r$ so that $ar = cq$.
	Then since $P$ is LCM there exists $k, r_1, q_1\in P$ such that $Pr\cap Pq = kP$ and $r_1r= q_1q = k$.
	Since $ar = cq$ is an element of $Pk$, there exists $k_1\in P$ such that $ar = cq = k_1k$.
	Thus $ar = k_1r_1r$ and hence $a = k_1r_1$.
	So $\gamma = (k_1r_1rpx, rpx)$, which is an element of both $\Delta^{r_1}$ and $\Delta_{rp}$.
	So we have $(\Delta_p\cap \Delta^q)_r\subseteq \Delta_{rp}\cap\Delta^{r_1}$.
	
	To show $\Delta_{rp}\cap\Delta^{r_1}\subseteq(\Delta_p\cap \Delta^q)_r$ in the case of a nonempty intersection, take $\gamma \in \Delta_{rp}\cap\Delta^{r_1}$. 
	Then $\gamma = (brpx,rpx) = (cr_1y,y)$ for some $b,x,c,y\in P$, which implies $y = rpx$ so that $brpx = cr_1rpx = cq_1qpx$ and hence $br = cq_1q$.
	Thus $\gamma = (brpx,rpx)$ with $br = (cq_1)q$, implying $\gamma \in (\Delta_p\cap \Delta^q)_r$.
	
	\ref{it6:constructibleLCM} Similar to \ref{it5:constructibleLCM}.
	
	The statement \eqref{eq:JPformLCM} at the end of the lemma now follows immediately, since points \ref{it3:constructibleLCM}--\ref{it6:constructibleLCM} imply that $\{\Delta_p\cap \Delta^q:p,q\in P\}\cup\{\emptyset\}$ is a subset of $\J(P)$ which is closed under intersections and the operations $Y\mapsto Y_p$ and $Y\mapsto Y^p$.
\end{proof}
We note that it is necessary to union with $\{\emptyset\}$ in \eqref{eq:JPformLCM} because it may be that the intersection of two sets of that type never results in the empty set.

Now we describe the inverse semigroup generated by the $v_p$ given in \eqref{eq:littlevpdef}.
\begin{lem} \label{lem:common_products} Let $P$ be an LCM monoid and let $p, q, r\in P$.
	\begin{enumerate}
		\item \label{it1:common_products}Suppose $pP\cap qP = rP$ and that $pp_1 = qq_1 = r$. Then
		\begin{align}
		v_p^*v_q &= v_{p_1}v_r^*v_q \label{eq:p*q1}\\
		&= v_p^*v_rv_{q_1}^*.\nonumber
		\end{align}
		Furthermore, if instead $pP\cap qP = \emptyset$, this product is zero.
		\item \label{it2:common_products} Suppose $Pp\cap Pq = Pr$ and that $p_2p = q_2q = r$. Then
		\begin{align}
		v_pv_q^* &= v_pv_r^*v_{q_2} \label{eq:pq*1}\\
		&= v_{p_2}^*v_rv_q^*\label{eq:pq*2}.
		\end{align}
		Furthermore, if instead $Pp\cap Pq = \emptyset$, this product is zero.
	\end{enumerate}
	\begin{proof}
		\ref{it1:common_products} If $pP\cap qP = \emptyset$ then $\Delta_p\cap \Delta_q = \emptyset$ by Lemma~\ref{lem:constructibleLCM}.\ref{it3:constructibleLCM}. 
		Hence the domain of $v_p^*$ does not intersect the range of $v_q$, so $v_p^*v_q = 0$.
		
		If $pP\cap qP = rP$ and that $pp_1 = qq_1 = r$, then Lemma~\ref{lem:constructibleLCM}.\ref{it3:constructibleLCM} implies $\Delta_p\cap\Delta_q = \Delta_r$ and so
		\begin{align*}
		v_p^*v_q &= v_p^*v_pv_p^*v_qv_q^*v_q\\
		&= v_p^*\Id_{\Delta_p}\Id_{\Delta_q}v_q\\
		&= v_p^*\Id_{\Delta_r}v_q\\
		&= v_p^*v_rv_r^*v_q\\
		&= v_p^*v_pv_{p_1}v_r^*v_q\\
		&= v_p^*v_pv_{p_1}v_{p_1}^*v_{p_1}v_r^*v_q &\text{since }v_{p_1}v_{p_1}^*v_{p_1}= v_{p_1}\\
		&= v_{p_1}v_{p_1}^*v_p^*v_pv_{p_1}v_r^*v_q &\text{since idempotents commute}\\
		&= v_{p_1}v_{pp_1}^*v_{pp_1}v_r^*v_q\\
		&= v_{p_1}v_r^*v_q&\text{since }pp_1 = r\text{ and }v_r^* = v_r^*v_rv_r^*.
		\end{align*}
		This establishes the first equality. 
		For the second, 
		\begin{align*}
		v_p^*v_q &= v_p^*v_rv_r^*v_q&\text{as above}\\
		&= v_p^*v_rv_{q_1}^*v_q^*v_q\\
		&= v_p^*v_rv_{q_1}^*v_{q_1}v_{q_1}^*v_q^*v_q &\text{since }v_{q_1}^*v_{q_1}v_{q_1}^*= v_{q_1}^*\\
		&= v_p^*v_rv_{q_1}^*v_q^*v_qv_{q_1}v_{q_1}^* &\text{since idempotents commute}\\
		&= v_p^*v_rv_{qq_1}^*v_{qq_1}v_{q_1}^*\\
		&=v_p^*v_rv_{q_1}^*&\text{since }qp_1 = r\text{ and }v_r = v_rv_r^*v_r.
		\end{align*}
		\ref{it2:common_products} These calculations are very similar to those in \ref{it1:common_products} and are left to the reader.	
	\end{proof}
\end{lem}
\begin{prop}\label{prop:sLCMform}
	Let $P$ be an LCM monoid, and let $\Irl(P)$ be as in \eqref{eq:s}. Then
	\begin{equation}\label{eq:sLCMform}
	\Irl(P) = \{v_pv_q^*v_r :  p, q, r\in P, q\in rP\cap Pp \}\cup\{0\}.
	\end{equation}
	Furthermore, we have that
	\[
	E(\Irl(P)) = \{v_pv_{qp}^*v_q :  p, q\in P\}\cup\{0\}.
	\]
\end{prop}
\begin{proof}
	Let $C$ denote the right hand side of \eqref{eq:sLCMform}. Then $C\subseteq E(\Irl(P))$ is trivial, because $\Irl(P)$ is generated by the $v_p$. 
	
	We show $C \subseteq E(\Irl(P))$ by showing that the given elements are closed under product and inverse, and hence form an inverse semigroup containing $v_p$ for each $p$ ($v_p\in \Irl(P)$ because $v_p = v_pv_p^*v_p$). 
	Since $\Irl(P)$ is the smallest such inverse semigroup, we will be done.

	Take $p,q,r\in P$ with $q\in rP\cap Pp$. 
	Then there exist $r_1, p_1\in P$ such that $q = rr_1 = p_1p$.
	We calculate
	\begin{align*}
	(v_pv_q^*v_r)^* &= v_r^*v_qv_p^*\\
	&= v_{r_1}v_q^*v_qv_p^* &\text{by }\eqref{eq:p*q1}\\
	&= v_{r_1}v_q^*v_qv_q^*v_{p_1}&\text{by }\eqref{eq:pq*1}\\
	&= v_{r_1}v_q^*v_{p_1}
	\end{align*}
	and so the right hand side of \eqref{eq:sLCMform} is closed under taking inverses.
	
	To show it is closed under taking products, take $p, q, r, a, b, c\in P$ such that $q \in rP\cap Pp$ and $b\in cP\cap Pa$. 
	Then there exist $r_1, p_1, a_1, c_1\in P$ such that $ q = rr_1 = p_1p$ and $b = cc_1 = a_1a$. 
	If the product $(v_pv_q^*v_r)(v_av_b^*v_c)$ is zero we are done, so at every step in the calculation below we will assume the product is nonzero, i.e. that $raP\cap qP = kP$ and $Pra\cap Pb = Pl$. We have
	\begin{align*}
	(v_pv_q^*v_r)(v_av_b^*v_c) &= v_pv_q^*v_{ra}v_b^*v_c \\
	&= v_p(v_{q_1}v_k^*v_{ra})v_b^*v_c &\text{ by \eqref{eq:p*q1} with } raa_2 = qq_1 = k\\
	&= v_{pq_1}v_k^*(v_{ra}v_l^*v_{b_1})v_c &\text{ by \eqref{eq:pq*2} with }r_2ra = b_1b = l\\
	&= v_{pq_1}v_{a_2}^*v_{ra}^*v_{ra}v_{ra}^*v_{r_2}^*v_{b_1c}\\
	&= v_{pq_1}v_{a_2}^*v_{ra}^*v_{r_2}^*v_{b_1c}\\
	&= v_{pq_1}v_{r_2raa_2}^*v_{b_1c} 	  
	\end{align*}
	for some $a_2, q_1, k, r_2, b_1, l\in P$. 
	Furthermore, since
	\[
	rr_2aa_2 = r_2qq_1 = r_2p_1pq_1 \in Ppq_1,
	\]
	\[
	rr_2aa_2 = b_1ba_2 = b_1cc_1a_2\in b_1cP,
	\]
	the product is of the form given in \eqref{eq:sLCMform}, so we have proven the first statement.
	
	Let $s = v_pv_q^*v_r$ for $p,q,r\in P$ with $q\in Pp\cap rP$, so that $q = p_1p = rr_1$ for some $r_1, p_1\in P$. Then
	\begin{align*}
	ss^* &= v_pv_q^*v_rv_r^*v_qv_p^*\\
	&= \Id_{((\Delta_r)^q)_p}&\text{by }\eqref{eq:vIdv}\\
	& = \Id_{(\Delta_1\cap \Delta^q)_p}&\text{by }\ref{it6:constructibleLCM}\text{ with }q1 = rr_1\\
	& = \Id_{\Delta_p\cap \Delta^{p_1}}&\text{by }\ref{it5:constructibleLCM}\text{ with }1q = p_1p\\
	&= v_pv_p^*v_{p_1}^*v_p\\
	&= v_pv_{p_1p}^*v_p.
	\end{align*}
	Every idempotent is of the form $ss^*$ for some $s\in \Irl(P)$, so we are done.
\end{proof}

We now show that the form of the elements of $\Irl(P)$ given in \eqref{eq:sLCMform} is essentially unique.
\begin{lem}\label{lem:s-elements}
	Let $P$ be an LCM monoid, and suppose that  $q\in Pp\cap rP$ and $b\in Pa\cap cP$. Then $v_pv_q^*v_r = v_av_b^*v_c$ if and only if there exist invertible elements $u, v\in U(P)$ such that $p = au$, $q = vbu$, and $r = vc$. 
\end{lem}
\begin{proof}
	Take $p, q, r, a, b, c\in P$ such that $q \in rP\cap Pp$ and $b\in cP\cap Pa$. 
	Then there exist $r_1, p_1, a_1, c_1\in P$ such that $ q = rr_1 = p_1p$ and $b = cc_1 = a_1a$. 
	Suppose that the maps $v_pv_q^*v_r$ and $v_av_b^*v_c$ are equal.
	Then since 
	\[
	v_pv_q^*v_r(q, r_1) = (q, p) = v_av_b^*v_c(q, r_1),
	\]
	there must exist $u,v\in P$ such that $q = vbu$, $au = p$ and $c_1u = r_1$.
	Similarly, since 
	\[
	v_av_b^*v_c(b, c_1) = (b,a) = v_pv_q^*v_r(b, c_1),
	\]
	there must exist $x,y\in P$ such that $b = yqx$, $px = a$, and $r_1x = c_1$.
	Since 
	\[
	a = px = aux, \hspace{1cm} r_1 = c_1u = r_1xu,
	\]
	cancellativity gives us that $ux = 1 = xu$.
	Furthermore, we have $q = vbu = vyqxu = vyq$ which implies that $yv = 1 = vy$.
	So $u, v$ are invertible elements of $P$ and $p = au, q = vbu$, and $r = vc$.
	
	To get the other direction, clearly if such invertible elements exist, then $	v_pv_q^*v_r = v_{au}v_{vbu}^*v_{vc} = v_av_uv_u^*v_bv_v^*v_vv_c = v_av_b^*v_c.$
\end{proof}

Lemma~\ref{lem:s-elements} allows us to give an abstract characterization of $\Irl(P)$.

\begin{prop}\label{prop:SSPiso}
	Let $P$ be an LCM monoid, and consider the equivalence relation on $P\times P \times P$ given by 
	\begin{equation}\label{eq:equivdef}
	(p,q,r)\sim (a,b,c) \hspace{0.5cm} \iff \exists u, v\in U(P) \text{ such that }p = au, q = vbu, r = vc
	\end{equation}
	and let $[p,q,r]$ denote the equivalence class of $(p,q,r)$ under this relation. 
	Then the set 
	\begin{equation}\label{eq:SP}
	\s_{P} = \{[p,q,r]  :  p, q, r\in P, q \in rP\cap Pp\}\cup\{0\}
	\end{equation}
	is an inverse semigroup when given the operations
	\[
	[p,q,r]^* = [r_1, q, p_1]\hspace{1cm}\text{where } q = rr_1 = p_1p,
	\]
	and 
	\begin{equation}\label{eq:SPproduct}
	[p,q,r][a,b,c] = \begin{cases}[pq_1, r_1raa_1, b_1c]&\text{if }raP\cap qP = kP; raa_1 = qq_1 = k\\&\text{and }Pra\cap Pb = Pl; r_1ra = b_1b = l\\0&\text{otherwise}\end{cases},
	\end{equation}
	The map $v_pv_q^*v_p \mapsto [p,q,r]$ and $0\mapsto 0$ is an isomorphism of inverse semigroups between $\Irl(P)$ and $\s_{P}$.  
	The set of idempotents of this inverse semigroup is given by
	\[
	E(\s_{P}) = \{[p,qp,q]\in\s_{P} :  q, p\in P\}\cup\{0\}.
	\]
\end{prop}
\begin{proof}
	All the statements follow from Lemma~\ref{lem:s-elements} and the calculations in the proof of Proposition~\ref{prop:sLCMform}.
\end{proof}

From now on we will work with elements in the form \eqref{eq:SP}.
\begin{lem}\label{lem:SPidempotents}
	Let $P$ be an LCM monoid, let $\s_P$ be as in \eqref{eq:SP}, and let $p, q, a, b\in P$. Then
	\begin{equation}\label{eq:idempotentmultiplication}
	[p,qp,q][a, ba, b] = \begin{cases}
	[r,sr,s] &\text{if }rP= pP\cap aP\text{ and }Ps = Pb\cap Pq\\
	0&\text{otherwise}
	\end{cases}.
	\end{equation}
	In particular, we have
	\[
	[p,qp,q] \leqslant [a,ba,b] \hspace{0.5cm}\iff\hspace{0.5cm} pP\subseteq aP \text{ and }Pq\subseteq Pb.
	\]
\end{lem}
\begin{proof}
	To verify \eqref{eq:idempotentmultiplication}, we calculate
	\begin{align*}
	[p,qp,q][a,ba,b] &= [p,p,1][1,q,q][a,a,1][1,b,b]\\
	&=[p,p,1][a,a,1][1,b,b][1,q,q]\\
	&=\begin{cases}
	[r,r,1][1,s,s] &\text{if }rP= pP\cap aP\text{ and }Ps = Pb\cap Pq\\
	0&\text{otherwise}
	\end{cases}\\
	&=\begin{cases}
	[r,sr,s] &\text{if }rP= pP\cap aP\text{ and }Ps = Pb\cap Pq\\
	0&\text{otherwise},
	\end{cases}
	\end{align*}
	where the third line is by Lemma~\ref{lem:constructibleLCM} and Proposition~\ref{prop:SSPiso}.
	Now we suppose $[p,qp,q] \leqslant [a,ba,b]$, that is $[p,qp,q] [a,ba,b] = [p,pq,q]$.
	Then by \eqref{eq:idempotentmultiplication}, we have that $pP = pP\cap aP$ and $Pq = Pq\cap Pb$, i.e. $pP\subseteq aP$ and $Pq\subseteq Pb$.
	Conversely, if $pP\subseteq aP$ and $Pq\subseteq Pb$, then one easily sees by \eqref{eq:idempotentmultiplication} that $[p,qp,q][a,ba,b]= [p,qp,q]$.
\end{proof}

It will frequently be convenient to use the following shorthand notation for often-used elements of $\s_{P}$: 
\begin{equation*}
[p] := [p,p,p]
\end{equation*}
This corresponds to $v_p$ above, and elements of this form generate $\s_{P}$.
\begin{lem}
	For an LCM monoid $P$ and $[p,q,r]\in \s_{P}$, we have
	\begin{equation*}
	[p,q,r][p,q,r]^* = [p,p_1p, p_1], \hspace{1cm} [p,q,r]^*[p,q,r] = [r_1, rr_1, r]
	\end{equation*}
	where $q = p_1p = rr_1$. In addition, 	for all $p,q\in P$ we have
	\begin{equation*}
	[p]^* = [p,p,p]^* = [1,p,1], \hspace{1cm}[p][q] = [pq],\hspace{1cm} [p]^*[q]^* = [qp]^*.
	\end{equation*}

\end{lem}
\begin{proof}
	Left to the reader.
\end{proof}
\begin{lem}\label{lem:E*unitary}
	Let $P$ be an LCM monoid, and let $\s_{P}$ be as in \eqref{eq:SP}. Then $\s_{P}$ is $E^*$-unitary. 
\end{lem}
\begin{proof}
	We must show that for $s\in \s_{P}$ and $e\in E(\s_{P})\setminus \{0\}$, $se = e$ implies $s$ is an idempotent.
	Let $s = [p,q,r]$ and suppose we have such an $e$.
	Since $se = e$, we must have $e\leqslant s^*s$, so $e = [b, cb, c]$ for some $b\in r_1P$ and $c\in Pr$ where $rr_1 = q = p_1p$. 
	Hence, $b = r_1b_1$ and $c = c_1r$.
	Calculating $se$, we have
	\begin{align*}
	se &= [p,q,r][b,cb,c]\\
	&= [p][q]^*[r][b][b]^*[c]^*[c]\\
	&= [p][q]^*[r][c]^*[c][b][b]^*\\
	&= [p][q]^*[r][c_1r]^*[c_1r][b][b]^*\\
	&= [p][q]^*[r][c_1r]^*[c_1r][b][b]^*\\
	&= [p][q]^*[r][r]^*[c_1]^*[c_1][r][b][b]^* &\text{since }
	[c_1r]^*[c_1r] = [r]^*[c_1]^*[c_1][r]\\
	&= [p][q]^*[c_1]^*[c_1][r][r]^*[r][b][b]^*&\text{because idempotents commute}\\
	&= [p][q]^*[c_1]^*[c_1][r][b][b]^*\\
	&= [p][r_1]^*[r]^*[c_1]^*[c_1][r][b][b]^*\\
	&= [p][r_1]^*[c]^*[c][b][b]^*\\
	&= [p][r_1]^*[b][b]^*[c]^*[c]&\text{because idempotents commute}\\
	&= [p][r_1]^*[r_1][b_1][b_1]^*[r_1]^*[c]^*[c]\\
	&= [p][b_1][b_1]^*[r_1]^*[r_1][r_1]^*[c]^*[c]\\
	&= [p][b_1][b_1]^*[r_1]^*[c]^*[c]\\
	&= [pb_1, cr_1b_1,c], 
	\end{align*}
	where the last line is because $cr_1 = c_1rr_1 = c_1p_1p \implies cr_1b_1\in Ppb_1\cap cP$. Now if $se =e$, the above element is an idempotent. 
	Hence $cpb_1 = cr_1b_1$, whence cancellativity implies that $p = r_1$. 
	Thus $[p,q,r] = [p,rr_1, r] = [p,rp,r] \in E(\s_{P})$.
\end{proof}

\begin{lem}\label{lem:C*spanningsets}
	Let $P$ be an LCM monoid.
	Then $\{S_pS_q^*S_r : p,q,r\in P, q\in Pp\cap rP\}\cup\{0\}$ is closed under multiplication and adjoint, and so forms an inverse semigroup of partial isometries in $\cpp$. In particular, its span is dense in $\cpp$ and the span of 
	$\{J_pJ_q^*J_r : p,q,r\in P, q\in Pp\cap rP\}$ is dense in $\ctspp$.
\end{lem}
\begin{proof}
	It is enough to show that $\{S_pS_q^*S_r : p,q,r\in P, q\in Pp\cap rP\}\cup\{0\}$ is closed under products, because it contains each $S_p$ and $S_p^*$. But the same calculations as in the proofs of Lemma~\ref{lem:common_products} and Proposition~\ref{prop:sLCMform} hold because 
	\[
	S_pS_p^*S_qS_q^*= S_pf_{\Delta}S_p^*S_qf_{\Delta}S_q^* = f_{\Delta_p}f_{\Delta_q}  = f_{\Delta_p\cap\Delta_q} = \begin{cases}
	f_{\Delta_r} & pP\cap qP = rP\\0 & pP\cap qP = \emptyset
	\end{cases},
	\]
	\[
	S_p^*S_pS_q^*S_q= S_p^*f_{\Delta}S_pS_q^*f_{\Delta}S_q = f_{\Delta^p}f_{\Delta^q}  = f_{\Delta^p\cap\Delta^q} = \begin{cases}
	f_{\Delta^r} & Pp\cap Pq = Pr\\0 & Pp\cap Pq = \emptyset
	\end{cases}.
	\] 
	All that was needed in those proofs was the above relations and the fact that idempotents commute, but all the idempotents there were of the form $\Id_Y$ for a constructible set $Y$, and the corresponding $f_Y$ commute by \ref{def:cstarmonoid}.\ref{it2:universal_def}. 
	That the same holds in $\ctspp$ follows from the universal property of $\cpp$. 
\end{proof}

As mentioned in the introduction, the choice of left regular representation over the right affects Li's construction. We can now show that in the LCM case, our construction puts the left and right multiplication on equal footing, similar to \cite[Corollary~3.5]{Tol17}. In what follows, $P^{\text{op}}$ denotes the opposite semigroup of $P$, which has the same elements as $P$ with multiplication
\[
p\cdot q = qp.
\]
It is clear that $P$ is LCM if and only if $P^{\text{op}}$ is.
\begin{prop}\label{prop:iso_opp}
	Let $P$ be an LCM monoid. Then $\cpp\cong C_{\text{ts}}^*(\pop)$.
\end{prop}
\begin{proof}
	Let $\{S_p\}_{p\in P} \cup \{f_Y\}_{Y\in \J(P)}$ and $\{T_p\}_{p\in \pop} \cup \{g_Z\}_{Z\in \J(\pop)}$ be the universal generating sets for $\cpp$ and $C_{\text{ts}}^*(\pop)$ respectively. Also write
	\begin{align*}
	\Delta &= \{(bx,x): b,x\in P\}, &\Gamma & = \{(c\cdot y, y): c,y\in \pop\}. 
	\end{align*}
	Define $h: \J(P)\to \J(\pop)$ by $h(\Delta_p\cap \Delta^q) = \Gamma^p\cap \Gamma _q$ and $h(\emptyset) = \emptyset$. We claim that the sets $\{T_p^*\}_{p\in P}$ and $\{g_{h(Y)}\}_{Y\in\J(P)}$ satisfy \ref{it1:universal_def}--\ref{it5:universal_def} in Definition~\ref{def:cstarmonoid}. Points \ref{it1:universal_def} and \ref{it3:universal_def} are straightforward, and \ref{it5:universal_def} is similar to \ref{it4:universal_def}. We prove \ref{it2:universal_def} and \ref{it4:universal_def}.
	
	For \ref{it2:universal_def}, if $Y, Z\in \J(P)$, with $Y= \Delta_p\cap \Delta^q$, $Z=\Delta_a\cap \Delta^b $, then 
	\[
	g_{h(Y)}g_{h(Z)} = g_{\Gamma^p\cap \Gamma_q}g_{\Gamma^a\cap \Gamma_b} = g_{(\Gamma^p\cap \Gamma^a)\cap(\Gamma_q\cap\Gamma_b)}.
	\]
	If either $P\cdot p\cap P\cdot a$ or $q\cdot P\cap b\cdot P$ is empty, then one of $pP\cap aP$ or $Pq\cap Pb$ is empty, so both sides are zero. On the other hand, if $P\cdot p\cap P\cdot a = P\cdot k$ and $q\cdot P\cap b\cdot P = \ell\cdot P$, then
	\[
	g_{h(Y)}g_{h(Z)} = g_{\Gamma^k\cap\Gamma_\ell} = g_{h(\Delta_k\cap \Delta^\ell)} = g_{h(Y\cap Z)}.
	\]
	For \ref{it4:universal_def}, let $p,q,r\in P$ and let $Y = \Delta_p\cap \Delta^q$. If $Pr\cap Pq = Pk$ for some $k\in P$ with $r_1r = q_1q = k$ then Lemma~\ref{lem:constructibleLCM} implies $(\Delta_p\cap \Delta^q)_r = \Delta_{rp}\cap \Delta^{r_1}$. Since we also have $r\cdot r_1 = q\cdot q_1 = k$, it also implies $(\Gamma^p\cap \Gamma_q)^r = \Gamma^{p\cdot r} \cap \Gamma_{r_1}$. Thus we have
	\begin{align*}
	T^*_pg_{h(Y)}(T_p^*)^* &= T^*_pg_{\Gamma^p\cap \Gamma_q}T_p = g_{\Gamma^{p\cdot r} \cap \Gamma_{r_1}} = g_{h(\Delta_{rp}\cap \Delta^{r_1})} = g_{h((\Delta_p\cap \Delta^q)_r)}\\
	&= g_{h(Y_r)}.
	\end{align*}
	Furthermore, if $Pr\cap Pq= \emptyset$, then the above calculation and Lemma~\ref{lem:constructibleLCM} imply that both sides of Definition~\ref{def:cstarmonoid}.\ref{it4:universal_def} are zero.
	
	Since $\{T_p^*\}_{p\in P}$ and $\{g_{h(Y)}\}_{Y\in\J(P)}$ satisfy \ref{it1:universal_def}--\ref{it5:universal_def} in Definition~\ref{def:cstarmonoid}, we can find a $*$-homomorphism $\Phi: \cpp\to C_{\text{ts}}^*(\pop)$ such that $\Phi(S_p) = T^*_p$ and $\Phi(f_Y) = g_{h(Y)}$. This argument can be repeated (because $(\pop)^{\text{op}} = P$) giving us a $*$-homomorphism $\Psi: C_{\text{ts}}^*(\pop)\to \cpp$ such that $\Psi(T_p) = S^*_p$ and $\Psi(g_Y) = f_{h^{-1}(Y)}$. Since $\Phi\circ \Psi$ and $\Psi\circ \Phi$ are the identity on the respective generating sets, both $\Phi$ and $\Psi$ must be isomorphisms.
\end{proof}

\subsection{Actions of inverse semigroups on their spectra and the associated groupoids}
In this section we recall the definitions of the spectrum and tight spectrum of a semilattice.
We also recall the definitions of the universal and tight groupoid of an inverse semigroup.
The discussion here attempts to summarize the important points of \cite{Ex08}---see there for a more detailed exposition.
For references on \'etale groupoids, see \cite{R80} and \cite{Sim20}.

Let $E$ be a semilattice, or equivalently a commutative inverse semigroup where every element is idempotent. 
We assume that $E$ has a bottom element 0. 
A {\em filter} in $E$ is a nonempty proper subset $\xi\subseteq E$ which is
\begin{itemize}
	\item {\em upwards closed} i.e. $e\in \xi$ and $fe = e$ implies $f\in \xi$ and
	\item {\em downwards directed} i.e. $e,f\in\xi$ implies $ef\in \xi$.
\end{itemize}
We let $\Ef$ denote the set of filters in $E$.
We identify the power set of $E$ with the product space $\{0,1\}^E$, and give $\Ef\subseteq \{0,1\}^E$ the subspace topology.
With this topology, $\Ef$ is called the {\em spectrum} of $E$.

Given $e\in E$ and $F\fs E$ the set
\[
U(e, F) = \{\xi\in \Ef : e\in \xi, \xi\cap F = \emptyset\}
\]
is a clopen subset of $\Ef$, and sets of this type generate the topology on $\Ef$. 

A filter is called an {\em ultrafilter} if it is not properly contained in another filter. 
The subspace of ultrafilters is denoted $\Eu\subseteq \Ef$, and its closure is denoted $\overline{\Eu} = \Et$ and is called the {\em tight spectrum} of $E$.

Let $S$ be an inverse semigroup with idempotent semilattice $E$ and let $X$ be a topological space.
An {\em action} of $S$ on $X$ is a pair $\theta = (\{\theta_s\}_{s\in S},\{D_e\}_{e\in E})$ where $D_e\subseteq X$ is open for all $e\in E$, $\theta_s: D_{s^*s}\to D_{ss^*}$ is a homeomorphism for all $s\in S$, $\theta_s\circ \theta_t = \theta_{st}$ for all $s,t\in S$ and $\theta_s^{-1} = \theta_{s^*}$. 
We also insist that $\theta_0$ is the empty map and $\cup D_e = X$.
When $\theta$ is an action of $S$ on $X$ we write $\theta:S\curvearrowright X$.

Given an action $\theta: S\curvearrowright X$, one puts an equivalence relation on the set $\{(s,x)\in S\times X: x\in D_{s^*s}\}$ stating $(s,x)\sim (t,y)$ if and only if $x = y$ and there exists $e\in E$ such that $x\in D_e$ and $se = te$.
We write $[s,t]$ for the equivalence class of $(s,t)$.
Then the {\em transformation groupoid} for $\theta$ is the set of equivalence classes
\[
\g^\theta = \{[s,x] : s\in S, x\in D_{s^*s}\}
\]
with range, source, inverse, and partially defined product given by
\[
r[s,x] = \theta_s(x), \hspace{.3cm}d[s,x] = x,\hspace{.3cm} [s,x]^{-1} = [s^*, \theta_s(x)], \hspace{.3cm}[t, \theta_s(x)][s,x] = [ts,x]
\]
This is an \'etale groupoid when given the topology generated by sets of the form
\begin{equation}\label{eq:ISG_action_basic_open_set}
\Theta(s,U) = \{[s,x]\in \g^\theta: x\in U\}\hspace{1cm} s\in S, U\subseteq D_{s^*s} \text{ open}.
\end{equation}

An inverse semigroup acts naturally on its spectrum. 
If $S$ is an inverse semigroup with idempotent semilattice $E$, we define an action $\alpha: S\curvearrowright \Ef$ by
\[
D_{e} = \{\xi\in \Ef: e\in \xi\} = U(e, \emptyset), \hspace{0.2cm}
\alpha_s: D_{s^*s}\to D_{ss^*}, \hspace{0.2cm}
\alpha_s(\xi) = \{ses^*: e\in \xi\}^{\uparrow}
\]
where the superscript $\uparrow$ indicates the set of all elements above some element in the set. 
The transformation groupoid associated to $\alpha$ is called the {\em universal groupoid of $S$}.

The space of tight filters is invariant under this action, so we get an action $\alpha:S\curvearrowright\Et$.
The transformation groupoid for this action is called the {\em tight groupoid of $S$}.

If $E$ and $F$ are semilattices with zero, then $E\times F$ is a semilattice with pointwise meet (product).
Consider the equivalence relation $\sim$ on $E\times F$ given by
\[
(0,0) \sim (e, 0)
\sim (0, f) \hspace{1cm}\text{for all } e\in E, f\in F
\]
Then $\sim$ is easily seen to be a {\em congruence}, that is $as \sim at$ and $sa \sim ta$ whenever $s\sim t$ and $a\in E\times F$. 
We denote the set of equivalence classes $$E\times F/\sim :=E\times_0F$$ and denote $[(0, 0)]_\sim := 0$. Then we have
\[
E\times_0F = \{(e,f) :  e\in E\setminus\{0\}, f\in F\setminus\{0\}\}\cup\{0\}.
\]
which is a semilattice under the inherited  operation
\[
(e_1,f_1)(e_2, f_2) = \begin{cases}
(e_1e_2,f_1f_2)&\text{if }e_1e_2\neq 0\text{ and } f_1f_2\neq 0\\
0&\text{otherwise}
\end{cases}.
\]

\begin{lem}\label{lem:semilatticeproduct}
	Let $E$ and $F$ be semilattices each with top and bottom elements.
	Then there is a homeomorphism $\vp:\widehat{(E\times_0F)}_0 \to \widehat E_0\times \widehat F_0$ which sends ultrafilters onto ultrafilters. 
	In particular, the tight spectrum of $E\times_0F$ is homeomorphic to $\widehat E_{\text{tight}}\times \widehat F_{\text{tight}}$.
\end{lem}
\begin{proof}
	Since filters are by definition proper subsets, a filter in $E\times_0 F$ must be a subset of $E\times F$. 
	For any subset $U\subseteq E\times F$ we write
	\[
	U_l = \{e\in E: (e,f)\in U\text{ for some }f\}, \hspace{0.3cm}U_r = \{f\in F: (e,f)\in U\text{ for some }e\}.
	\]
	Note that if $U$ is a filter then $e\in U_l \iff (e,1)\in U$ and $f\in U_r \iff (1,f)\in U$ because filters are upwards closed.
	
	Now define $\vp: \widehat{(E\times_0F)}_0 \to \widehat E_0\times \widehat F_0$ by
	\[
	\vp(\xi) = (\xi_l, \xi_r),\hspace{1cm} \xi \in \widehat{(E\times_0F)}_0.
	\]
	It is clear that both $\xi_l$ and $\xi_r$ are filters, so $\vp$ is well-defined.
	
	To see that $\vp$ is injective, suppose $\vp(\xi) = \vp(\eta)$, so that $\xi_l = \eta_l$ and $\xi_r = \eta_r$. 
	Then $(e,f) \in \xi$ implies  $e\in \xi_l$ and $f\in \xi_r$. This implies that $ e\in \eta_l$ and $f\in \eta_r$, which gives $(e,1), (1,f)\in \eta$ and so $(e,f)\in \eta$. By a symmetric argument we get $\xi = \eta$. 
	
	To see that $\vp$ is surjective, take filters $\xi \subseteq E$ and $\eta\subseteq F$ and consider $\xi\times \eta\subseteq E\times_0 F$. 
	It straightforward to check that $\xi\times\eta$ is a filter, and that $\vp(\xi\times\eta) = (\xi, \eta)$.
	
	To show continuity, take $e\in E$, $Y\fs E$, $f\in F$ and $Z\fs F$ and consider the open set 
	\[
	U = U(e,Y)\times U(f, Z) = \{(\xi,\eta) \in \widehat{E}_0\times\widehat{F}_0 : e\in \xi\subseteq Y^c, f\in \eta\subseteq Z^c\}. \]
	Then we have
	\begin{align*}
	\vp^{-1}(U) &= \{\xi\times \eta \in \widehat{(E\times_0F)}_0 : e\in \xi\subseteq Y^c, f\in \eta\subseteq Z^c\}\\
	&= U((e,f), (Y\times \{1\})\cup (\{1\}\times Z)).
	\end{align*}
	To see the last equality, we have $\xi\times \eta\in \vp^{-1}(U)$ if and only if $e\in \xi$, $f\in \eta$, and $Y\cap \xi = \emptyset = Z\cap \eta$. 
	If $y\in Y$ then $y\notin \xi$ and so $(y,1)\notin \xi\times\eta$; we similarly see that $(1,z)\notin \xi\times\eta$ for all $z\in Z$. 
	Hence $\xi\times \eta\in U((e,f), (Y\times \{1\})\cup( \{1\}\times Z))$ and we have one containment.
	Conversely, if $\xi\times \eta \in U((e,f), (Y\times \{1\})\cup (\{1\}\times Z))$ we have that $e\in\xi, f\in\eta$, and $((Y\times \{1\})\cup (\{1\}\times Z))\cap \xi\times\eta =\emptyset$.
	If $y\in Y$, then $(y,1)\notin \xi\times\eta$ implies $y\notin \xi$, and so $\xi\in U(e, Y)$.
	We similarly have $\eta\in U(f,Z)$ and so $\xi\times\eta\in \vp^{-1}(U)$.
	This shows that $\vp$ is continuous.
	
	Now given a basic open set $U((e,f), Y)$ in the spectrum of $E\times_0 F$, it is similarly checked that $\vp(U((e,f), Y)) = U(e,Y_l)\times U(f,Y_r)$. 
	Hence $\vp$ is a homeomorphism.
	
	Finally, if $\xi\subseteq E\times_0 F$ is an ultrafilter, then $\xi_l$ and $\xi_r$ are clearly ultrafilters too. 
	Conversely, if $\xi\subseteq E$ and $\eta\subseteq F$ are ultrafilters, then $\xi\times \eta$ is as well.
	Since $\vp$ is a homeomorphism we have
	\begin{align*}
	\vp\left(\widehat{(E\times_0F)}_{\text{tight}}\right) &= \vp\left(\overline{\widehat{(E\times_0F)}_{\infty}}\right) = 
	\overline{\vp\left(\widehat{(E\times_0F)}_{\infty}\right)}
	 =
	\overline{\Eu \times \widehat F_\infty}\\
&=  \widehat E_{\text{tight}}\times \widehat F_{\text{tight}}.
	\end{align*}	
\end{proof}
\subsection{The action of $\s_{P}$ on its spectra}
In what follows, we let
\[
P_r=\{pP :  p\in P\}\cup\{\emptyset\}, \hspace{1cm}
P_l=\{Pp :  p\in P\}\cup\{\emptyset\}
\]
which are both semilattices under intersection (due to $P$ being LCM).

\begin{lem}\label{lem:idempotentiso}
	Let $P$ be an LCM monoid and let $\s_{P}$ be as in \eqref{eq:SP}. Then  $E(\s_{P})$ and $P_l\times_0 P_r$ are isomorphic as semilattices, via the map $\phi:E(\s_{P})\to P_l\times_0 P_r$ defined by
	\[
	\phi[p,qp,q]=(Pq,pP), \hspace{1cm} \phi(0) = 0.
	\]
\end{lem}
\begin{proof}
	To start, note that $\phi$ is well-defined: if $u,v\in U(P)$ we have
	\[
	\phi[pu,vqpu,vq] = (Pvq,puP) = (Pq,pP) = \phi[p,qp,q] \hspace{1cm}\forall p,q\in P.
	\]
	
	If $pP\cap aP = \emptyset$ or $Pq\cap Pb=\emptyset$, then $$\phi([p,qp,q][a,ba,b]) = \phi(0) = 0,$$ while $\phi[p,qp,q]\phi[a,ba,b] = 0$ as well. Otherwise, if they are both nonempty, say $rP= pP\cap aP$ and $Ps = Pb\cap Pq$, then 
	\begin{align*}
	\phi([p,qp,q][a,ba,b]) &= \phi[r, sr,s]& \text{Lemma}~\ref{lem:SPidempotents}\\
	&= (Ps,rP)\\
	&= (Pb\cap Pq,pP\cap aP)\\
	&= (Pq,pP)(Pb, aP)\\
	&= \phi[p,qp,q]\phi[a,ba,a].
	\end{align*}
	Surjectivity is clear, and if $\phi[p,qp,q] = \phi[a,ba,b]$ we have $aP=pP$ and $Pb = Pq$ which implies there exist $u,v\in U(P)$ such that $a= pu$ and $b=vq$, giving us that $[p,qp,q] = [a,ba,b]$. 
\end{proof}
\begin{rmk}We note that our definition of $\phi$ may seem strange given that up to this point idempotents have been written in the form $v_pv_p^*v_q^*v_q$.
	Since $v_pv_p^*$ corresponds to $pP$ and $v_q^*v_q$ to $Pq$, it might seem more natural to send this idempotent to $(pP, Pq)$.
	We switch the order for two reasons. 
	The first is so that the semilattice of principal left ideals is written on the left (and likewise for the right).
	The other is to make things more clear in Example~\ref{ex:freesemigroups}.
\end{rmk}
For $p\in P$, define
\begin{align}
D_p^l &= U(\{Pp\}, \emptyset) \subseteq \widehat{(P_l)}_0,\label{eq:dpl}\\
D_p^r &= U(\{pP\}, \emptyset) \subseteq \widehat{(P_r)}_0.\label{eq:dpr}
\end{align}
For $p\in P$ and any right ideal $X\subseteq P$, the set
\[
p^{-1}X = \{y\in P: py\in X\}.
\]
is also a right ideal.
If $X = qP$ for some $q\in P$, then 
\[
p^{-1}qP = \begin{cases} p_1P &\text{if }pP\cap qP = rP, pp_1 = qq_1 = r\\
\emptyset& \text{if }pP\cap qP = \emptyset\end{cases}.
\]
Similarly, if $Y\subseteq P$ is a left ideal, the set
\[
Yp^{-1} = \{x\in P: xp\in Y\}
\]
is also a left ideal. 
If $Y = Pq$ for some $q\in P$, then 
\[
Pqp^{-1} = \begin{cases} Pp_1 &\text{if }Pp\cap Pq = Pr, p_1p = q_1q = r\\
\emptyset& \text{if }Pp\cap Pq = \emptyset\end{cases}.
\]
We then define, for $p\in P$, the following maps
\begin{align*}
&R_p: D_1^r \to D_p^r,&&L_p: D_p^l\to D_1^l,\\
&R_p(\xi) = p\xi,&&L_p(\xi) = \xi p^{-1}.
\end{align*}
Since every filter contains 1, we have $D_1^l = \widehat{(P_l)}_0$ and $D_1^r= \widehat{(P_r)}_0$.
Then the intrinsic action of $\s_{P}$ on its spectrum, viewed through the homeomorphism given in Lemma~\ref{lem:semilatticeproduct} is given by 
\[
\theta_{[p]}: D_p^l \times D_1^r \to D_1^l \times D_p^r, \hspace{1cm}
\theta_{[p]}(\xi,\eta) = (\xi p^{-1}, p\eta),
\]
which implies $\theta_{[p]^*} = \theta_{[p]}^{-1}(\xi,\eta) = (\xi p, p^{-1}\eta)$.

For general elements $[p,q,r]\in \s_{P}$, since $[p,q,r]= [p][q]^*[r]$, the action is given by
\begin{equation}
\theta_{[p,q,r]}: D_{r}^l \times D_{r_1}^r \to D_{p_1}^l \times D_{p}^r, \hspace{1cm}\theta_{[p,q,r]}(\xi,\eta) = (\xi r^{-1}qp^{-1}, pq^{-1}r\eta),\label{eq:actiononspectrum}
\end{equation}
where $q = p_1p = rr_1$.

\section{The C*-algebras associated to $P$}

\subsection{C*-algebras associated to inverse semigroups}

To an inverse semigroup $S$ one may associate several C*-algebras.
Some are defined in terms of groupoids associated to $S$ and some using representations.
We recall their definitions here.

A {\em representation} of $S$ on a C*-algebra $A$ is a function $\pi:S\to A$ such that $\pi(st) = \pi(s)\pi(t)$ for all $s,t\in S$, $\pi(s^*) = \pi(s)^*$ for all $s\in S$ and $\pi(0) = 0$. 
The {\em universal C*-algebra of S}, denoted $C^*_u(S)$, is the universal C*-algebra for representations of $S$.
This means that there is a representation $\piu: S\to C_u^*(S)$ such that if $\pi:S\to A$ is any other representation, there exists a *-homomorphism $\vp:C^*_u(S)\to A$ such that $\vp\circ \piu = \pi$.
We call $\piu$ the {\em universal representation of $S$.}

There is a map $\Lambda: S \to \B(\ell^2(S))$ defined by
\[
\Lambda(s)\delta_t = \begin{cases}
\delta_{st} & \text{if }s^*st = t\\0&\text{otherwise}
\end{cases},
\]
which can be shown to be a representation of $S$.
The image of $\Lambda$ generates a C*-algebra $C^*_r(S)$, called the {\em reduced} C*-algebra of $S$.

For a semilattice $E$ we say that a set $C\subseteq E$ is a {\em cover} of $e\in E$ if $c\leqslant e$ for all $c\in C$ and for all $f\leqslant e$ there exists $c\in C$ such that $cf\neq 0$.
A representation $\pi$ of a {\em unital} semilattice is {\em tight} if whenever $C$ is a cover of $E$ we have $\bigvee_{c\in C}\pi(c) = \pi(e)$.
If $S$ is an inverse monoid, a unital representation of $S$ is tight if its restriction to $E(S)$ is.
Note that this is not the original definition of tight as given by Exel in \cite{Ex08}, but is equivalent in this setting, see \cite[Proposition~11.8]{Ex08}, \cite[Corollary~2.3]{DM14}, and \cite{Ex19}.

Then the {\em tight C*-algebra of $S$} \cite{Ex08}, denoted $\Ct(S)$, is universal for tight representations of $S$.
That is, there is a tight representation $\pit:S\to \Ct(S)$ and if $\pi:S\to A$ is any other tight representation, there exists a *-homomorphism $\vp:C^*_u(S)\to A$ such that $\vp\circ \pit = \pi$.
We call $\pit$ the {\em universal tight representation of $S$}.

These C*-algebras have realizations as groupoid C*-algebras.
We have that $C^*_u(S)\cong C^*(\g_u(S))$ and $\Ct(S)\cong C^*(\gt(S))$, and under these isomorphisms we have
\[
\piu(s) = 1_{\Theta(s, D_{s^*s})},\hspace{1cm}\pit(s)= 1_{\Theta(s, D_{s^*s}\cap \Et)}.
\]
\subsection{$\cpp$ as a groupoid C*-algebra}

\begin{theo}\label{thm:fulliso}
	Let $P$ be an LCM monoid, let $\s_P$ be as in \eqref{eq:SP}, and recall that $\g_u(\s_P)$ is the universal groupoid of $\s_P$. 
	Then 
	$$\cpp\cong C^*_u(\s_P)\cong C^*(\g_u(\s_P)).$$
\end{theo}
\begin{proof}
	As mentioned above, $C^*_u(\s_P)\cong C^*(\g_u(\s_P))$ is established in \cite{Ex08}. 
	We will obtain the first isomorphism using the universal properties of the algebras. 
	For $p,q\in P$ and $\Delta_p\cap \Delta^q\in \J(P)$ let 
	\[
	T_p = \piu([p]),\hspace{1cm}E_{\Delta_p\cap \Delta^q} = \piu([p,qp,q]),\hspace{0.5cm} E_\emptyset = 0.
	\] 
	We first notice that the latter is well-defined, since $\Delta_p = \Delta_a$ and $\Delta^q = \Delta^b$ if and only if $pP = aP$ and $Pq = Pb$, which implies $[p,qp,q] = [a,ba,b]$.
	We claim that these elements satisfy Definition~\ref{def:cstarmonoid}.
	That each $T_p$ is a partial isometry, each $E_Y$ is a projection, and that \ref{it1:universal_def} and \ref{it2:universal_def} in Definition~\ref{def:cstarmonoid} are satisfied is clear.
	Noticing that $\Delta = \Delta_1\cap \Delta^1$ shows that $E_\Delta = 1$, so we have \ref{it3:universal_def}.
	
	To show \ref{it4:universal_def}, we take $p,q,r\in P$. 
	If $Pr\cap Pq = Pk$ with $r_1r = q_1q = k$, then by Lemma~\ref{lem:constructibleLCM} we have
	\begin{align*}
	T_rE_{\Delta_p\cap\Delta^q}T_r^*&=\piu([r][p,qp,q][r]^*)\\
	&= \piu([r][p][p]^*[q]^*[q][r]^*[r][r]^*)\\
	&= \piu([r][p][p]^*[k]^*[k][r]^*)\\
	&= \piu([r][p][p]^*[r]^*[r_1]^*[r_1][r][r]^*) \\
	&= \piu([r][p][p]^*[r]^*[r_1]^*[r_1])\\
	&= \piu([rp,r_1rp,r_1])\\
	&= E_{\Delta_{rp}\cap \Delta^{r_1}}\\
	&= E_{(\Delta_p\cap\Delta^q)_r}.	
	\end{align*}
	The calculation for \ref{it5:universal_def} is similar.
	Hence by the universal property of $\cpp$ there exists a $*$-homomorphism $\Psi: \cpp \to C^*_u(\s_P)$ such that $\Psi(S_p) = T_p$ and $\Psi(e_Y) = E_Y$ for all $p\in P$ and $Y\in \J(P)$. 
	
	For the other direction, we claim that the map $\pi: \s_{P}\to \cpp$ given by
	\[
	\pi([p,q,r]) = S_pS_q^*S_r\hspace{1cm}\pi(0) = 0
	\]
	is a representation of $\s_{P}$.
	It is straightforward to check that $\pi$ is well-defined.
	Looking at Definition~\ref{def:cstarmonoid} and Proposition~\ref{prop:sLCMform} shows that the elements of $\{S_pS_q^*S_r:p,q,r\in P, q\in Pp\cap rP\}$ multiply in the same way as the elements of $\s_{P}$.
	The same arguments as in their proofs show that $\pi$ is a representation.
	Hence by the universal property there exists a $*$-homomorphism $\Phi:C^*_u(\s_{P})\to \cpp$ such that $\Phi(T_p) = S_p$ and $\Phi(E_Y) = e_Y$ for all $p\in P$ and $Y\in \J(P)$.
	Hence, $\Phi\circ\Psi = \id_{\cpp}$ and $\Psi\circ\Phi = \id_{C^*_u(\s_{P})}$ implying that $\Psi$ and $\Phi$ are isomorphisms. 
\end{proof}

\subsection{$\ctspp$ as a reduced groupoid C*-algebra}
\begin{lem}\label{lem:independent}
	Let $P$ be an LCM monoid and let $p,q, p_1, q_i\in P$ for $i = 1, \dots n$. 
	If $\Delta_p\cap \Delta^q = \cup_{i=1}^n\Delta_{p_i}\cap\Delta^{q_i}$, then there exists $i\in\{1,\dots n\}$ such that $\Delta_p = \Delta_{p_i}$ and $\Delta^q = \Delta^{q_i}$.
\end{lem}

In the words of \cite[Definition~2.26]{Li12}, $\J(P)$ is {\em independent}.
\begin{proof}
	We have that $(qp,p)\in \Delta_p\cap \Delta^q$, so $(qp,q)\in \Delta_{p_i}\cap \Delta^{q_i}$ for some $i\in\{1,\dots,n\}$.
	Since $(qp,p)\in \Delta_{p_i}\cap \Delta^{q_i}$, it must have the form $(bq_ip_ix, p_ix)$ for some $b,x\in P$, see \eqref{eq:DeltapDeltaq}.
	Thus $p = p_ix$ and $q = bq_i$, which implies $pP\subseteq p_iP$ and $Pq \subseteq Pq_{i}$.
	Lemma~\ref{lem:constructibleLCM} and its proof then imply that $\Delta_p\cap\Delta^q\subseteq  \Delta_{p_i}\cap \Delta^{q_i}$, and since the other containment is assumed we have equality.	 
\end{proof}
\begin{theo}\label{thm:reducediso}
	Let $P$ be an LCM monoid and let $\s_P$ be as in \eqref{eq:SP}. 
	Then $\ctspp\cong C^*_r(\s_{P})\cong C^*_r(\g_u(\s_{P}))$
\end{theo}
\begin{proof}
	Define an operator $T:\ell^2(\Delta)\to \ell^2(\s_{P})$ by
	\[
	T(\delta_x^{bx}) = \delta_{[x,bx,bx]}.
	\]
	It is straightforward to check that its adjoint is given by
	\[
	T^*(\delta_{[p,q,r]}) = \begin{cases}
	\delta_{pu}^{qu}&qu = r \text{ for some } u\in U(P)\\
	0&\text{otherwise}
	\end{cases}
	\]
	and that $T^*T = \id_{\ell^2(\Delta)}$, so that $T$ is an isometry.
	Now define $h:\B(\ell^2(\s_{P}))\to \B(\ell^2(\Delta))$ by $h(a) = T^*aT$.
	If we have $p,q,r\in P$ with $q\in Pp\cap rP$, then
	\begin{align*}
	h(\Lambda([p,q,r]))\delta_x^{bx}&=T^*\Lambda([p])\Lambda([q]^*)\Lambda([r])T\delta_x^{bx}\\
	&= T^*\Lambda([p])\Lambda([q]^*)\Lambda([r])\delta_{[x,bx,bx]}\\
	&= \begin{cases}T^*\Lambda([p])\Lambda([q]^*)\delta_{[rx,bx,bx]} &b\in Pr\\
	0&\text{otherwise}
	\end{cases}\\
	&= \begin{cases}T^*\Lambda([p])\delta_{[q_1,bx,bx]} &b\in Pr\text{ and }rx = qq_1\\
	0&\text{otherwise}
	\end{cases}\\
	&= \begin{cases}T^*\delta_{[pq_1,bx,bx]} &b\in Pr, rx = qq_1, \text{ and }bx\in Ppq_1\\
	0&\text{otherwise}
	\end{cases}\\
	&= \begin{cases}\delta_{pq_1}^{bx} &b\in Pr, rx = qq_1, \text{ and }bx\in Ppq_1\\
	0&\text{otherwise}
	\end{cases}\\
	&= J_pJ_q^*J_r\delta_x^{bx}.
	\end{align*}
	Hence restricted to the dense $*$-subalgebra generated by $\Lambda(\s_{P})$, $h$ is multiplicative and preserves adjoints, so is a $*$-homomorphism there. 
	As defined $h$ is continuous, and its image is a dense subalgebra of $\ctspp$, so $h$ extends to a $*$-homomorphism $h: C^*_r(\s_{P})\to \ctspp$. 
	This $*$-homomorphism must be surjective since $h(C^*_r(\s_{P}))$ is a C*-algebra, hence closed, and contains a dense subalgebra of $\ctspp$.
	
	To show injectivity, we use conditional expectations. 
	Let $E_\Delta:\B(\ell^2(\Delta))\to \ell^\infty(\Delta)$ be the canonical faithful conditional expectation determined by the formula $\left\langle E_\Delta(a) \delta^{bx}_x, \delta^{bx}_x\right\rangle = \left\langle a(\delta^{bx}_x), \delta^{bx}_x\right\rangle$.
	Here we are identifying $\ell^\infty(\Delta)$ with the subalgebra of $\B(\ell^2(\Delta))$ of operators determined by pointwise multiplication by bounded functions.
	We claim that 
	\[
	E_\Delta(J_pJ_q^*J_r) = \begin{cases}J_pJ_q^*J_r &q = rp\\0&\text{otherwise}.\end{cases}
	\]
	Indeed, from the definition of $E_\Delta$, we see that $E_\Delta(J_pJ_q^*J_r)$ will be zero unless $J_pJ_q^*J_r$ fixes some $\delta_x^{bx}$.
	This occurs when $x = pq_1$, where $rx = qq_1$.
	But then $qq_1 = rx = rpq_1$ which implies $q = rp$. 
	To finish the claim then we should show that if $q = rp$ then $E_\Delta(J_pJ_{rp}^*J_r) = J_pJ_{rp}^*J_r$, but this is immediate.
	
	Since $\s_{P}$ is E*-unitary, there is also a conditional expectation on $C_r^*(\s_{P})$ onto the commutative C*-algebra $D(\s_{P})$ generated by $\Lambda(E(\s_{P}))$ \cite[Proposition~3.7]{No14}. 
	It is given on generators by 
	$$E(\Lambda(s)) = \begin{cases}\Lambda(s)& s\in E(\s_{P})\\
	0&\text{otherwise}\end{cases}. $$
	A short calculation shows that $h\circ E = E_\Delta\circ h$. 
	
	Finally, if $h(a) = 0$, then $h(a^*a) = 0$, and so $E_\Delta(h(a^*a)) = 0$.
	Thus $h(E(a^*a)) = 0$, but \cite[Proposition~3.5]{No14} and Lemma~\ref{lem:independent} combine to show that $h$ is injective on the image of $E$, hence $E(a^*a) = 0$. 
	Since $E$ is faithful, $a=0$ so $h$ is injective. 
	This establishes the first isomorphism.
	
	The second isomorphism is standard, see \cite{Pa99} and \cite{No14}.
\end{proof}
\subsection{The boundary quotient}

The results of \cite{StLCM} suggest that the natural boundary quotient for $\cpp$ should be the tight C*-algebra of $\s_P$. 
Hence, we take this to be the {\em definition} of the boundary quotient.

\begin{defn}
	Let $P$ be an LCM monoid, and let $\s_{P}$ be as in \eqref{eq:SP}.
	We define the {\em boundary quotient} of $\cpp$, denoted $\qpp$, to be the tight C*-algebra of $\s_{P}$,
	\[
	\qpp:= \Ct(\s_{P}).
	\]
\end{defn}
Here we always have a conditional expectation on to the diagonal subalgebra.
\begin{prop}
	Let $P$ be an LCM monoid. Then the map $\varphi: \qpp\to \qpp$ defined on generators of $\qpp$ by
	\[
	\varphi(\pi_t([p,q,r])) = \begin{cases}\pi_t([p,rp,r]) &\text{if }q = rp\\ 0&\text{otherwise}\end{cases}
	\]
	extends to a conditional expectation onto the subalgebra of $\qpp$ generated by $\pi_t(E(\s_{P}))$.
\end{prop}
\begin{proof}
	By Lemma~\ref{lem:E*unitary}, the tight groupoid is Hausdorff. 
	Since $\gt(\s_{P})$ is second countable and \'etale, we know from \cite{R80} that there is a conditional expectation from $\Ct(\s_{P})$ to $C(\gt(\s_{P})^{(0)})$ which is given on $C_c(\gt(\s_{P}))$ by function restriction, $f\mapsto \left. f\right|_{\gt(\s_{P}^{(0)}}$.
	On the generators (which are elements of $C_c(\gt(\s_{P}))$), the given map $\varphi$ is exactly restriction to $\gt(\s_{P})^{(0)}= \Et(\s_{P})$, which is the C*-algebra generated by $\pi_t(E(\s_{P}))$.
\end{proof}

\begin{prop}
	Let $P$ be an LCM monoid, and suppose that $P$ embeds into an amenable group $G$. 
	Then $\cpp$ and $\qpp$ can be realized as partial crossed products of commutative C*-algebras by $G$, and hence are nuclear.
\end{prop}
\begin{proof}
	Let $\s_P$ be as in \eqref{eq:SP} and define
	\[
	\psi: \s_P^\times \to G, \hspace{1cm}
	\psi([p,q,r]) = pq^{-1}r.
	\]
	It is straightforward to check that $\psi$ is well-defined. 
	Suppose that we have $p,q,r, a,b,c\in P$ such that $[p,q,r][a,b,c]\neq 0$. 
	Then by \eqref{eq:SPproduct} there exist $k, a_1, q_1, l, r_1, b_1\in P$ such that $raP\cap qP = kP, Pra\cap Pb = Pl$, and 
	\begin{align}raa_1 &= qq_1 = k,\label{eq:cal1}\\
	r_1ra &= b_1b = l,\label{eq:cal2}
	\end{align} and $[p,q,r][a,b,c] =[pq_1, r_1raa_1, b_1c]$.
	Hence
	\begin{align*}
	\psi([p,q,r][a,b,c]) &= \psi[pq_1, r_1raa_1, b_1c]\\
	&= pq_1(r_1raa_1)^{-1}b_1c\\
	&= pq_1a_1^{-1}a^{-1}r^{-1}r_1^{-1}b_1c\\
	&= p(raa_1q^{-1})^{-1}r_1^{-1}b_1c\\
	&= pq^{-1}r_1^{-1}b_1c &\text{since }raa_1q^{-1} = q\text{ by }\eqref{eq:cal1}\\
	&= pq^{-1}rab^{-1}c &\text{since }r_1^{-1}b = rab^{-1} \text{ by }\eqref{eq:cal2}\\
	&=\psi[p,q,r]\psi[a,b,c]
	\end{align*}
	So $\psi$ is multiplicative away from zero.
	Furthermore, if  $\psi[p,q,r] = 1_G$, we have
	$q^{-1} = p^{-1}r^{-1}$ which implies $q = rp$, and so $[p,q,r]$ is an idempotent. 
	Hence $\psi$ is what is usually termed an {\em idempotent pure prehomomorphism} of the inverse semigroup $\s_P$, and so by \cite[Corollary~3.4]{Li17} (see also \cite{MS14}) both $\cpp = C^*(\s_P)$ and $\qpp = \Ct(\s_P)$ can be expressed as partial crossed products of commutative C*-algebras by $G$. 
	Since $G$ is amenable, the conclusion follows from \cite[Corollary~3.4]{Li17} (see also \cite{ExBook}).
\end{proof}

\section{Examples}\label{sec:examples}

\subsection{Free Semigroups}\label{ex:freesemigroups}

We retain notation from Example \ref{ex:FreeSemigroups} above. Let $X$ be a finite set and let $X^*$ be the free semigroup over $X$. We show that the boundary quotient $\mathcal{Q}_{\text{ts}}(X^*)$ is isomorphic to the crossed product associated to the two-sided full shift over $X$.

For $x\in X^*\cup X^\NN$ and $m,n\in \NN$ with $m<n$, define
\begin{align*}
x_{[m,n]} &:= x_mx_{m+1}\cdots x_n \label{eq:nprefixdef},\\
x_{[n]} &:= x_{[1,n]}.\nonumber
\end{align*}
For $\alpha\in X^*$, we also let
\[
\overleftarrow{\alpha} := \alpha_{|\alpha|}\alpha_{|\alpha|-1}\cdots\alpha_2\alpha_1.
\]
If $x\in X^\NN$, the set
\[
\xi_x = \{x_{[n]}X^*: n\in \NN \} \cup\{X^*\}
\]
is an ultrafilter in the semilattice $X^*_r$ of principal right ideals. Likewise,
\[
\eta_x = \{X^*\overleftarrow{x_{[n]}} : n\in \NN\} \cup\{X^*\}
\]
is an ultrafilter in the semilattice $X^*_l$ of principal left ideals. Furthermore, the map $x\mapsto \xi_x$ (resp. $x\mapsto \eta_x$) is a homeomorphism from $X^\NN$ onto $\Eu(X^*_r) = \Et(X^*_r)$ (resp. onto $\Eu(X^*_l) = \Et(X^*_l)$).
Referring to \eqref{eq:dpl} and \eqref{eq:dpr}, we have
\[
D_\alpha^l = \{\eta_{\overleftarrow{\alpha}x} : x\in X^\NN\}, \hspace{1cm}D_\alpha^r = \{\xi_{\alpha x} : x\in X^\NN\}.
\]
If $\alpha\in X^*$, $x\in \Et(X_r^*)$, and $y\in \Et(X_l^*)$ then 
\[
\alpha \xi_x = \xi_{\alpha x},\hspace{1cm}\eta_y \overleftarrow{\alpha} = \eta_{\alpha y}.
\]
We view $X^\NN \times X^\NN$ as the Cantor space of bi-infinite sequences in $X$; and so for $x, y\in X^\NN$ we use the identification
\begin{equation}\label{eq:twosidedsequences}
(x,y) = \dots x_3x_2x_1.y_1y_2y_3\dots 
\end{equation}
where we are dropping the 0th entry for convenience. 
For $\alpha, \beta\in X^*$, let
\begin{equation*}\label{eq:Calphabeta}
C(\alpha, \beta) = \{(\alpha x, \beta y): x,y\in X^\NN\}.
\end{equation*}
Sets of this form generate the product topology on $X^\NN\times X^\NN$, and they are clopen in this topology.

In identifying $\Et(X^*_l)\times \Et(X_r^*)$ with $X^\NN\times X^\NN$, we get an action of $\s_{X^*}$ on $X^\NN\times X^\NN$. 
Since $X^*$ has no invertible elements, a given $[\alpha,\beta,\gamma]\in \s_{X^*}$ is a one-element equivalence class. 
For such an element, we have that $\beta = \alpha_1\alpha = \gamma\gamma_1$ for some $\alpha_1, \gamma_1\in X^*$.
Then referring to \eqref{eq:actiononspectrum} the action of $\s_{X^*}$ on $X^\NN\times X^\NN$ is given by 
\[
\theta_{[\alpha,\beta,\gamma]}: C(\overleftarrow{\gamma}, \gamma_1) \to C(\overleftarrow{\alpha_1},\alpha), \hspace{1cm}\theta_{[\alpha,\beta,\gamma]}(\overleftarrow{\gamma}x, \gamma_1 y) = (\overleftarrow{\alpha_1} x, \alpha y).
\]
When viewed with the identification given in \eqref{eq:twosidedsequences} the map is given by
\[
\theta_{[\alpha,\beta,\gamma]}(\dots x_2x_1 \overbrace{\gamma . \gamma_1}^\beta y_1y_2\dots) = \dots x_2x_1 \overbrace{\alpha_1 . \alpha}^\beta y_1y_2\dots.
\] 
In words, an element $[\alpha,\beta,\gamma]$ being in $\s_{X^*}$ indicates that $\gamma$ is a prefix of $\beta$ and $\alpha$ is a suffix of $\beta$.
Then $\theta_{[\alpha,\beta,\gamma]}$ acts on two-sided infinite sequences which have the word $\beta$ at the origin situated so that the prefix $\gamma$ is to the left of the origin. 
The map $\theta_{[\alpha,\beta,\gamma]}$ then shifts this sequence so that the suffix $\alpha$ is to the right of the origin.

\begin{lem}\label{lem:cocycle} The map $h: \s_{X^*}^\times\to \ZZ$ given by 
	\[
	h[\alpha,\beta,\gamma] = |\beta|-|\alpha|-|\gamma|
	\]
	is an idempotent-pure prehomomorphism. 
\end{lem}
\begin{proof}
	Let $p,q,r,a,b,c\in X^*$ and suppose that $[p,q,r][a,b,c]\neq 0$. 
	Then \eqref{eq:SPproduct} implies there exist $a_1, q_1, r_1, b_1\in X^*$ such that $raa_1 = qq_1$ and $r_1ra=b_1b$ and $[p,q,r][a,b,c] = [pq_1, r_1raa_1, b_1c]$.
	Then since $|r_1ra| - |b_1| = |b|$, $|a_1| = |qq_1| - |r|-|a|$, and $|qq_1| - |q_1| = |q|$ we have
	\begin{align*}
	h([p,q,r][a,b,c]) &= |r_1raa_1| - |pq_1| - |b_1c|\\
	&= |r_1ra| + |a_1|- |pq_1| - |b_1| - |c|\\
	&= |b| + |a_1| - |pq_1| - |c|\\
	&= |b| + |qq_1| - |r| -|a| - |p| - |q_1| - |c|\\
	&= |b| + |q| - |r| - |p| -|c| - |a|\\
	&=h[p,q,r]+ h[a,b,c].	
	\end{align*} 
	Furthermore, if $h[p,q,r] = 0$ we have that $|q| = |p| + |r|$ and together with the fact that $q\in X^*p\cap rX^*$ we have that $q = rp$ so that $[p,q,r]$ is an idempotent.
\end{proof}

The left shift map $\sigma: X^\NN\to X^\NN$ is the the homeomorphism given by
\begin{equation*}
\sigma(x,y) = (y_1x, y_2y_3\cdots) = \dots x_3x_2x_1y_1.y_2y_3\dots .
\end{equation*}
Lemma~\ref{lem:cocycle} and the discussion before it show that 
\begin{equation*}
\theta_{s}(x,y) = \sigma^{h(s)}(x,y), \hspace{1cm}s\in \s_{X^*}^\times.
\end{equation*}
Let $\g^\sigma$ be the transformation groupoid associated to the $\ZZ$ action on $X^\NN \times X^\NN$, so that
\begin{equation}
\g^\sigma = \{(n, (x,y)) : n\in \ZZ, x,y\in X^\NN\}.
\end{equation}

\begin{theo}\label{thm:crossedproductfullshift}
	Let $X$ be a finite set and let $X^*$ be the free monoid on $X$. 
	Then the tight groupoid associated to $\s_{X^*}$ is isomorphic to $\g^\sigma$. 
	In particular, $$\mathcal{Q}_{\text{ts}}(X^*)\cong C(X^\NN\times X^\NN)\rtimes_\sigma \ZZ.$$	
\end{theo}
\begin{proof}
	Define $\Phi: \gt(\s_{X^*})\to \g^\sigma$ by
	\[
	\Phi([s, (x,y)]) = (h(s), (x,y)).
	\]
	We first show $\Phi$ is well-defined.
	Suppose that $[s, (x,y)] = [t, (x,y)]$ which means there is an idempotent $e$ such that $se = te$. 
	Since $h(e) = 0$ for every idempotent $e$ we have
	\[
	h(s) = h(s) + h(e) = h(se) = h(te) = h(t),
	\]
	which implies $\Phi([s,(x,y)]) = \Phi([t,(x,y)])$.
	
	Given $[t,\theta_s(x,y)],[s,(x,y)]\in \gt(\s_{X^*})$ we have
	\begin{align*}
	\Phi([t,\theta_s(x,y)][s,(x,y)]) &= \Phi([ts, (x,y)]) \\
	&=(h(ts), (x,y))\\
	&=(h(t)+ h(s), (x,y))\\
	&=(h(t), \sigma^{h(s)}(x,y))(h(s), (x,y))\\
	&=(h(t), \theta_s(x,y))(h(s), (x,y))\\
	&=\Phi([t,\theta_s(x,y)])\Phi([s,(x,y)]),\\
	\Phi([s,(x,y)]^{-1}) &= \Phi([s^*, \theta_s(x,y)])\\
	&= (h(s^*), \sigma^{h(s)}(x,y))\\
	&= (-h(s), \sigma^{h(s)}(x,y))\\
	&= (h(s), (x,y))^{-1}\\
	&= \Phi([s,(x,y)])^{-1},
	\end{align*}
	which shows that $\Phi$ is a groupoid homomorphism. 
	
	To show that $\Phi$ is injective, we suppose that $\Phi([s,(x,y)]) = \Phi([t,(z,w)])$, which implies $(x,y) = (z,w)$ and $h(s) = h(t)$.
	Since the domains of $\theta_s$ and $\theta_t$ contain a common ultrafilter, this implies $s^*st^*t\neq 0$ and so $st^*$ and $ts^*$ are both nonzero. 
	But then by Lemma~\ref{lem:cocycle} we have $h(st^*) = h(s) - h(t) = 0$ which implies $st^*$ is an idempotent (and is hence equal to its adjoint $ts^*$).
	We then have
	\[
	st^*t s^*s = ts^*ts^*s = ts^*s = tt^*ts^*s = ts^*st^*t
	\] 
	so taking $e = t^*ts^*s = s^*st^*t$ in the transformation groupoid definition gives $[s,(x,y)] = [t,(x,y)]$.
	
	To show that $\Phi$ is surjective, let $g = (n, (x,y))\in \g^\sigma$. 
	If $n = 0$, then $\Phi(1, (x,y)) = g$.
	If $n>0$, then $\Phi([\epsilon, x_{[n]}, \epsilon], (x,y)) = (|x_{[n]})|, (x,y)) = g$.
	If $n<0$, then $\Phi([x_{[n]},x_{[n]},x_{[n]}], (x,y)) = (-|x_{[n]})|, (x,y)) = g$.
	Hence $\Phi$ is surjective.
	
	Finally, if $\Theta(s,U)$ is a basic open set in $\gt(\s_{X^*})$, we have $\Phi(\Theta(s,U)) = \{h(s)\}\times U$ which is clearly open in $\g^\sigma$, so that $\Phi$ is an open map.
	On the other hand if $U\subseteq X^\NN\times X^\NN$ is open and $n\in\ZZ$, we have
	\[
	\Phi^{-1}(\{n\}\times U) = \bigcup_{s\in h^{-1}(n)} D_s\cap U
	\]
	which is open. 
	Hence $\Phi$ is a homeomorphism and we are done.	
\end{proof}
\begin{rmk}
	The existence of an idempotent-pure prehomomorphism into $\ZZ$ implies that $\gt(\s_{X^*})$ can be expressed as a partial action groupoid $\ZZ\ltimes \Et(\s_{X^*})$, see \cite[Corollary~3.4]{Li17} and \cite{MS14}.
	In this case the action ends up being a full action, because the domains of the elements of $h^{-1}(n)$ have union equal to the whole of $\Et(\s_{X^*})$.
\end{rmk}

\begin{rmk}
	Recall from \cite[Section~8.2]{Li13} that the boundary quotient $\Q(X^*)$ of Li's $C^*(X^*)$ is canonically isomorphic to $\mathcal{O}_{|X|}$, which is purely infinite and simple. 
	In contrast, our construction applied to the free semigroup gives something much different---the crossed product $C(X^\NN\times X^\NN)\rtimes \ZZ$ is far from simple (as the full shift has many periodic points and is hence not minimal).
	In addition, the full shift has many invariant measures which in turn gives $C(X^\NN\times X^\NN)\rtimes \ZZ$ many traces, making it stably finite.
\end{rmk}

\subsection{Self-similar actions}
To a self-similar action $(G,X)$ as defined in Example~\ref{ex:ssg}, Nekrashevych associated a C*-algebra $\ogx$ universal for a set of isometries $\{s_x:x\in X\}$ and a set of unitaries $\{u_g: g\in G\}$ satisfying
\begin{enumerate}
	\item[(SS1)] $\sum_{x\in X}s_xs_x^* = 1$ and $s_x^*s_y = 0 $ for $x\neq y$,
	\item[(SS2)] $u_gu_h = u_{gh}$ for all $g,h\in G$,
	\item[(SS3)] $u_g^* = u_g^{-1}$ for all $g\in G$,
	\item[(SS4)] $u_gs_x = s_{g\cdot x} u_{\left.g\right|_x}$ for all $g\in G$, $x\in X$.
\end{enumerate}

Let $(G,X)$ be a pseudo-free self-similar action.
To make what follows more readable, we will write
\[
P:= X^*\bowtie G.
\]
By Lemma~\ref{lem:SSGLCM}, $P$ is an LCM monoid. 
In what follows, we also assume that $(X,G)$ is recurrent.
Although this is not needed to make $P$ an LCM monoid, it does seem to be satisfied by many important examples.
The group of invertible elements is $U(P) = \{(\epsilon, g): g\in G\}$ and readily identified with $G$.

By Lemma~\ref{lem:recurrent} we have that $P_l$ is linearly ordered by inclusion; this has some important consequences for the tight C*-algebra.
Firstly, its space of ultrafilters is a singleton, so the space of ultrafilters of $E(\s_{P})$ can be identified with $\widehat{(P_r)}_{\text{tight}} \cong X^\NN$. 
Secondly, given two nonempty elements of $P_l$, one is {\em dense} in the other (recall that $e$ is dense in $f$ if $e\leqslant f$ and $g\leqslant f$ implies $ge\neq 0$.)
This means that $[1,p,p]$ is dense in $[1,1,1]$ for all $p\in P$ and so by \cite[Proposition 2.10]{Ex09}, 
\[
\pi[1,p,p] = \pi(1)\text{ for any tight representation }\pi:\s_{P}\to A\text{ in a C*-algebra }A.
\]
So the tight C*-algebra of $\s_{P}$ does not see its action on (space of ultrafilters of) the left ideals, leaving only its action on the (space of ultrafilters of) the right ideals.
It is this action which gives Nekrashevych's $\ogx$. 
Evidence is mounting that $\Ct(\s_{P})\cong \ogx$, and this indeed ends up being the case. 
In the remainder of this paper we prove this isomorphism by showing that the underlying groupoids are isomorphic, although it is possible to prove it using the universal properties of the algebras\footnote{We took this approach in an earlier preprint version of this work; see \href{https://arxiv.org/abs/2001.00156}{2001.00156} v3 on the arXiv.}. 

\begin{lem}\label{lem:ssgIdem}
	Let $(G,X)$ be a pseudo-free and recurrent self-similar action, let $P = X^*\bowtie G$ and let $\s_{P}$ be as in \ref{eq:sLCMform}.
	Then 
	\[E(\s_{P}) = \{[(\alpha,1_G), (\alpha,1_G), 1][1, (\beta,1_G), (\beta,1_G)]: \alpha, \beta\in X^*\}\cup\{0\}.
	\]
	Furthermore, we have that 
	\begin{equation}\label{eq:ssg_left_idempotent}
	[(\alpha,1_G), (\alpha,1_G), 1] = [(\delta,1_G), (\delta,1_G), 1]\iff \alpha = \delta,
	\end{equation}
	\begin{equation}\label{eq:ssg_right_idempotent}
	[1, (\beta,1_G), (\beta,1_G)] = [1, (\gamma,1_G), (\gamma,1_G)] \iff |\beta|=|\gamma|.
	\end{equation}
\end{lem}
\begin{proof}
	For the first statement, it is enough to show that each idempotent can be written in the given form, since each given element is clearly an idempotent. Suppose $$[(\alpha,g), (\beta, h)(\alpha,g), (\beta,h)] = [(\alpha,g), (\alpha, g), 1][1, (\beta,h),(\beta,h)] \in E(\s_{P}).$$
	Since $(G,X)$ is recurrent we can find $k\in G$ such that $\left. k\right|_\beta = h^{-1}$.
	Since $U(X^*\bowtie G) = \{\epsilon\}\times G$ we have
	\[
	[(\alpha,g), (\alpha, g), 1] = [(\alpha,g)(\epsilon,g^{-1}), (\alpha, g)(\epsilon, g^{-1}), 1] = [(\alpha, 1_G), (\alpha, 1_G), 1)],
	\]
	\[
	[1, (\beta,h),(\beta,h)] = [1, (\epsilon, k)(\beta,h),(\epsilon, k)(\beta,h)] = [1, (k\cdot\beta, 1_G), (k\cdot\beta, 1_G)].
	\]
	For \eqref{eq:ssg_left_idempotent}, if $[(\alpha,1_G), (\alpha,1_G), 1] = [(\delta,1_G), (\delta,1_G), 1]$ then in particular we can find $g\in G$ such that $(\alpha,g) = (\delta, 1_G)$, and so $\alpha = \delta$. The other implication is clear.
	
	For \eqref{eq:ssg_right_idempotent}, if $[1, (\beta,1_G), (\beta,1_G)] = [1, (\gamma,1_G), (\gamma,1_G)]$, then we can find $g,h\in G$ with $(\epsilon, g) = 1$, $(h\cdot\beta, \left.h\right|_\beta g) = (\gamma,1_G)$ and $(h\cdot\beta, \left.h\right|_\beta) = (\gamma,1_G)$; these equations imply $|\beta| = |\gamma|$. On the other hand, if $|\beta| = |\gamma|$ we use the fact that $(G, X)$ is recurrent to find $h\in G$ such that $h\cdot\beta = \gamma$ and $\left.h\right|_\beta = 1_G$ so that $[1, (\beta,1_G), (\beta,1_G)] = [1, (\epsilon, h)(\beta,1_G), (\epsilon, h)(\beta,1_G)] = [1, (\gamma,1_G), (\gamma,1_G)]$. 
\end{proof}
In light of the above, we will write
\begin{align}
A_\alpha &:= [(\alpha, 1_G), (\alpha, 1_G), 1]&\alpha\in X^*,\nonumber\\
B_n &:= [1, (\beta,1_G), (\beta, 1_G)]&\beta\in X^n\label{eq:Bndef},
\end{align}
so that $E(\s_{P}) = \{A_\alpha B_n: \alpha\in X^*, n\geq 0\}\cup\{0\}$. Note also that this means $[(\beta, g)]^*[(\beta, g)] = B_{|\beta|}$ for all $g\in G$, and that $B_nB_m = B_{\max\{m,n\}}$ for all $m,n\geq 0$. 
\begin{lem}\label{lem:ssgsimplification}
	Let $(G,X)$ be a pseudo-free and recurrent self-similar action, let $P = X^*\bowtie G$, and let $\s_{P}$ be as in \eqref{eq:SP}.
	Then 
	\[
	\s_{P} = \{[(\alpha,g)][(\gamma,1_G)]^*B_n: \alpha,\gamma\in X^*, g\in G, n\geq 0\}\cup\{0\}.
	\]
	Furthermore, $[(\alpha,g)][(\gamma,1_G)]^*B_n = [(\delta,h)][(\sigma,1_G)]^*B_m$ if and only if $\alpha = \delta$, $g=h$, $\gamma = \sigma$ and $n = m$. 
\end{lem}
\begin{proof}
	We first show that every element can be written in the form 
	\begin{equation}\label{eq:initialForm}
	[(\alpha, g), (\beta\gamma, 1_G), (\beta, 1_G)]\text{ for some }\alpha,\beta,\gamma\in X^*, |\beta\gamma|\geq|\alpha|.
	\end{equation} 
	Note that $|\beta\gamma|\geq |\alpha|$ is equivalent to saying $(\beta\gamma, 1_G) \in P(\alpha, g)$ by Lemma~\ref{lem:SSGLCM}.
	
	Take $[(\alpha, g), (\beta, h),(\gamma, k)]\in \s_{P}$. 
	Taking $u = (\epsilon, h^{-1})$ and $v = 1$ in \eqref{eq:equivdef} and renaming variables shows we can assume, without loss of generality, that $h = 1_G$. 
	We can also assume that $\gamma$ is a prefix of $\beta$, since $(\beta, 1_G)P\subseteq (\gamma,k)P$.
	Hence up to renaming variables, our generic element of $\s_{P}$ can be taken in the form $[(\alpha, g), (\beta\gamma, 1_G), (\beta, k)]$.
	Since $(G,X)$ is recurrent, there exists $a\in G$ such that $\left.a\right|_\beta = k^{-1}$.
	Then taking $u = (\epsilon, (\left.a\right|_{\beta\gamma})^{-1})$ and $v = (\epsilon,a)$ in \eqref{eq:equivdef} gives that
	\begin{align*}
	[(\alpha, g), (\beta\gamma, 1_G), (\beta, k)] &= [(\alpha, g)u,v(\beta\gamma, 1_G)u, v(\beta, k)]\\
	&= [(\alpha, g(\left.a\right|_{\beta\gamma})^{-1}), (a\cdot(\beta\gamma), 1_G), (a\cdot\beta, 1_G)],
	\end{align*}
	which is of the form \eqref{eq:initialForm} since $a\cdot \beta$ is a prefix of $a\cdot(\beta\gamma)$.
	
	Now suppose we have an element in the form \eqref{eq:initialForm}. We have two possibilities.	If $|\gamma|\geq |\alpha|$, then $[(\alpha, g), (\gamma, 1_G), 1]\in \s_P$ and so we have 
	\begin{align*}
	[(\alpha, g)][(\gamma, 1_G)]^*B_{|\beta|} &= [(\alpha, g), (\gamma, 1_G), 1][1, (\beta, 1_G), (\beta, 1_G)]\\ &= [(\alpha, g), (\beta\gamma, 1_G), (\beta, 1_G)]
	\end{align*}
	with $|\beta\gamma|\geq |\gamma|\geq |\alpha|$. 
	
	On the other hand, if $|\gamma|<|\alpha|$ we cannot proceed as above because the element $[(\alpha, g), (\gamma, 1_G), 1]$ is not in $\s_P$. However, one can check using \eqref{eq:SPproduct} that
	\begin{equation*}
	[(\alpha, g)][(\gamma, 1_G)]^* = [(\alpha, g), (\alpha, g), (\alpha, g)][1,(\gamma, 1_G), 1] = [(\alpha, g), (\alpha, g), (\xi, k)],
	\end{equation*}
	where $\xi$ is the prefix of $\alpha$ of length $|\alpha|-|\gamma|$ and $k\in G$ satisfies $\left.k\right|_\gamma = 1_G$ and $k\cdot\gamma$ is the suffix of $\alpha$ of length $|\gamma|$ (we can always find such a $k$ because $(G, X)$ is recurrent). Note that $|\beta\gamma|\geq |\alpha|$ and $|\xi|+|\gamma| = |\alpha|$ implies $|\beta|\geq |\xi|$. Thus we can compute 
	\begin{align*}
	[(\alpha, g), (\alpha, g), (\gamma, k)][1, (\beta, 1_G), (\beta, 1_G)] &= [(\alpha, g), (\beta, 1_G)(\gamma, 1_G), (\beta, 1_G)]\\ &= [(\alpha, g), (\beta\gamma, 1_G), (\beta, 1_G)]
	\end{align*}
	using \eqref{eq:SPproduct} with $p = q = (\alpha, g)$, $r_1r = b = c = (\beta, 1_G)$ and $a = q_1 = b_1 = 1$. So in either case we have written a generic element of $\s_P$ in the claimed form.
	
	To prove the second statement, if we have that $[(\alpha,g), (\beta\gamma,1_G), (\beta, 1_G)] = [(\delta,h), (\sigma\tau,1_G), (\sigma\tau, 1_G)]$, then there exists $s, t\in G$ such that 
	\[(\alpha, g) = (\delta, h)(\epsilon, s), \hspace{0.2cm}(\beta\gamma,1_G) = (\epsilon, t)(\sigma\tau,1_G)(\epsilon, s), \hspace{0.2cm} (\beta, 1_G)= (\epsilon, t)(\sigma, 1_G).
	\]
	The last equation implies $t\cdot\sigma = \beta$ (establishing $|\sigma| = |\beta|$) and $\left.t\right|_\sigma = 1_G$. Together with the second equation this implies $\tau = \gamma$ and $s = 1_G$, and hence by the first equation we have $\alpha = \delta$ and $g = h$. 
	
	On the other hand, since $(G, X)$ is recurrent if $|\sigma| = |\beta|$ we can find $t\in G$ such that $t\cdot \sigma = \beta$ and $\left.t\right|_\sigma = 1_G$, and so the previous calculation shows $[(\alpha,g), (\beta\gamma,1_G), (\beta, 1_G)] = [(\delta,h), (\sigma\tau,1_G), (\sigma\tau, 1_G)]$.
\end{proof}

Since the set $\{B_n: n\geq 0\}$ is linearly ordered, it admits only one ultrafilter. Hence, referring to \eqref{eq:actiononspectrum} and its section, we can identify the space of ultrafilters of $E(\s_P)$ with $X^\NN$. The action of $\s_P$ viewed through this identification is then given by 
\[
\theta_{[(\alpha, g)]}: X^\NN \to C(\alpha), \hspace{1cm}\theta_{[(\alpha,g)]} (x) = \alpha(g\cdot x).
\]
with $B_n = \Id_{X^\NN}$ for all $n$. More generally, we have
\[
\theta_{[(\alpha,g)][(\gamma,1_G)]^*B_n}: C(\gamma) \to C(\alpha), \hspace{1cm}\theta_{[(\alpha,g)][(\gamma,1_G)]^*B_n}(\gamma x) = \alpha(g\cdot x).
\]
It was proven in \cite{EP17} that $\ogx$ arises from the tight groupoid of another inverse semigroup associated to $(G, X)$. We show that the two groupoids are isomorphic, which will then imply $\Ct(\s_P)\cong \ogx$. 

The inverse semigroup defined in \cite{EP17} is given by 
\begin{equation}\label{eq:sgxdef}
\sgx = \{(\alpha, g, \beta):\alpha, \beta\in X^*, g\in G\}\cup\{0\}
\end{equation}
with inverse given by $(\alpha, g, \beta)^* = (\beta, g^{-1}, \alpha)$ and product given by 
\begin{equation}\label{eq:sgxproduct}
(\alpha, g, \gamma)(\delta, k, \sigma) = \begin{cases}
(\alpha(g\cdot \gamma_1), \left.g\right|_{\gamma_1}, \sigma)& \delta = \gamma\gamma_1\\
(\alpha, g\left(\left.k^{-1}\right|_{\delta_1}\right)^{-1}, \sigma(k^{-1}\cdot\delta_1)) & \gamma = \delta\delta_1\\
0 & \delta X^*\cap \gamma X^* = \emptyset
\end{cases},
\end{equation}
and all products involving 0 equal to 0. 

\begin{lem}\label{lem:semigrouphomomorphism}
	Let $(G,X)$ be a pseudo-free and recurrent self-similar action, let $P = X^*\bowtie G$, let $\s_{P}$ be as in \eqref{eq:SP}, and let $\sgx$ be as in \eqref{eq:sgxdef}. Then the map $\vp: \s_P\to \sgx$ defined by
	\[
	\vp([(\alpha,g)][(\gamma,1_G)]^*B_n) = (\alpha, g, \beta), \hspace{1cm} \vp(0) = 0
	\]
	is an inverse semigroup homomorphism.
\end{lem}
\begin{proof}
	Lemma~\ref{lem:ssgsimplification} shows that $\vp$ is well-defined. First we claim that
	\begin{equation}\label{eq:movingidempotentsforward}
	B_n[(\alpha, g)]= [(\alpha, g)]B_{n+ |\alpha|}, \hspace{1cm}(\alpha, g)\in P, n\geq 0.
	\end{equation}
	Picking some $\beta\in X^n$ we have
	\begin{align*}
	B_n[(\alpha, g)] & = [1, (\beta, 1_G), (\beta, 1_G)][(\alpha, g), (\alpha, g), (\alpha,g)]\\
	&=[(\alpha,g), (\beta\alpha, g), (\beta\alpha,g)] \\
	& = [(\alpha, g)][(\beta\alpha, g)]^*[(\beta\alpha, g)]\\
	&=  [(\alpha, g)]B_{n+|\alpha|}
	\end{align*}
	with the second line obtained from \eqref{eq:SPproduct} with $p , a_1, r_1 = 1, a,b,c, q_1 = (\alpha, g)$, and $q = r = b_1 = (\beta, 1_G)$. Since we always have $[(\alpha,g)] =  [(\alpha,g)]B_{|\alpha|}$, this implies that
	\begin{equation*}
	[(\alpha,g)]B_m = B_{\max\{0, m - |\alpha|\}}[(\alpha,g)], \hspace{1cm}(\alpha, g)\in P, m\geq 0.
	\end{equation*}
	In particular,
	\begin{align}\label{eq:movingidempotentsbackward}
	B_m	[(\alpha,g)]^* &= ([(\alpha,g)]B_m)^* = (B_{\max\{0, m - |\alpha|\}}[(\alpha,g)])^*\\ &= [(\alpha,g)]^*B_{\max\{0, m - |\alpha|\}}.\nonumber
	\end{align}
	Now we show
	\begin{equation} \label{eq:inside_of_sgx_product}
	[(\gamma, 1_G)]^*[(\delta, 1_G)] = \begin{cases}
	[(\gamma_1, 1_G)]B_{|\delta|}& \delta = \gamma\gamma_1\\
	[(\delta_1, 1_G)]^*B_{|\delta|}& \gamma = \delta\delta_1\\
	0 &\text{otherwise}
	\end{cases},\hspace{0.2cm}\gamma,\delta\in X^*, g\in G.
	\end{equation}
	If $\delta = \gamma\gamma_1$ we have
	\begin{align*}
	[(\gamma, 1_G)]^*[(\delta, 1_G)] &= [(\gamma, 1_G)]^*[(\gamma, 1_G)][(\gamma_1, 1_G)] = B_{|\gamma|}[(\gamma_1, 1_G)]\\& = [(\gamma_1, 1_G)]B_{|\gamma| + |\gamma_1|} = [(\gamma_1, 1_G)]B_{|\delta|}
	\end{align*}
	by \eqref{eq:movingidempotentsforward}. If $\gamma = \delta\delta_1$ then
	\[
	[(\gamma, 1_G)]^*[(\delta, 1_G)] = [(\delta_1, 1_G)]^*[(\delta, 1_G)]^*[(\delta, 1_G)] = [(\delta_1, 1_G)]^*B_{|\delta|}
	\]
	by \eqref{eq:movingidempotentsforward}. If neither of the above are true, then $A_\gamma A_\delta = 0$ which implies $[(\gamma, 1_G)]^*[(\delta, 1_G)] = [(\gamma, 1_G)]^*A_\gamma A_\delta[(\delta, 1_G)] = 0$. This establishes \eqref{eq:inside_of_sgx_product}.
	
	We now give a formula for the product of two arbitrary elements. For $\alpha, \gamma, \delta, \sigma\in X^*$, $g, k\in G$  and $m,n\geq 0$, let $M = \max\{n + |\delta|-|\sigma|, m\}$, and label our two arbitrary elements
	$$C = [(\alpha,g)][(\gamma, 1_G)]^*B_n, \hspace{0.5cm}D = [(\delta,k)][(\sigma, 1_G)]^*B_m.$$ Then we have
	\begin{equation}\label{eq:SPssgproduct}
	CD = \begin{cases}
	[(\alpha(g\cdot\gamma_1), \left.g\right|_{\gamma_1})][(\sigma, 1_G)]^*B_M&\delta = \gamma\gamma_1\\
	[(\alpha,g\left(\left.k^{-1}\right|_{\delta_1}\right)^{-1})][(\sigma(k^{-1}\cdot\delta_1), 1_G)]^*B_M&\gamma = \delta\delta_1\\
	0 &\text{otherwise}	
	\end{cases}.
	\end{equation}
	We show the case $\gamma = \delta\delta_1$; the other case is similar but more straightforward. Note that $B_n[(\delta,k)][(\sigma, 1_G)]^*B_m =[(\delta,k)][(\sigma, 1_G)]^*B_{\max\{n+|\delta|-|\sigma|, 0\}}B_m$, so this is equal to $[(\sigma, 1_G)]^*B_M$. Thus
	\begin{align*}
	CD&= [(\alpha,g)][(\gamma, 1_G)]^*[(\delta,k)][(\sigma, 1_G)]^*B_M\\
	&= [(\alpha,g)][(\delta_1, 1_G)]^*B_{|\delta|}[(\epsilon, k)][(\sigma, 1_G)]^*B_M&\text{by }\eqref{eq:inside_of_sgx_product}\\
	& = [(\alpha,g)][(\delta_1, k)]^*[(\epsilon, k^{-1})]^*[(\sigma, 1_G)]^*B_{|\delta| - |\sigma|}B_M&\text{by} \eqref{eq:movingidempotentsforward}\eqref{eq:movingidempotentsbackward}\\
	& = [(\alpha,g)][(\sigma, 1_G)(\epsilon, k^{-1})(\delta_1, 1_G)]^*B_M\\
	&= [(\alpha,g)][(\sigma k^{-1}\cdot\delta_1,\left.k^{-1}\right|_{\delta_1})]^*B_M\\
	& = [(\alpha,g)][(\epsilon,\left.k^{-1}\right|_{\delta_1})]^*[(\sigma k^{-1}\cdot\delta_1,1_G)]^*B_M\\
	& = [(\alpha,g)][(\epsilon,(\left.k^{-1}\right|_{\delta_1})^{-1})][(\sigma k^{-1}\cdot\delta_1,1_G)]^*B_M\\
	&=[(\alpha,g\left.k^{-1}\right|_{\delta_1}^{-1})][(\sigma(k^{-1}\cdot\delta_1), 1_G)]^*B_M.
	\end{align*}
	In the above we have used the easy-to-check fact that  $[(\epsilon, k)]^* = [(\epsilon, k^{-1})]$. One also uses this to show
	\begin{align*}
	([(\alpha,g)][(\gamma, 1_G)]^*)^* &= [(\gamma, 1_G)][(\alpha,g)]^* = [(\gamma, 1_G)][(\epsilon, g)]^*[(\alpha,1_G)]^* \\&= [(\gamma, g^{-1})][(\alpha,1_G)]^*.
	\end{align*}
	Comparing this and \eqref{eq:SPssgproduct} with \eqref{eq:sgxproduct} and the inverse above it shows that $\vp$ is an inverse semigroup homomorphism.	
\end{proof}
\begin{theo}\label{thm:SSGboundaryiso}
	Let $(G,X)$ be a pseudo-free and recurrent self-similar action, let $P = X^*\bowtie G$, let $\s_{P}$ be as in \eqref{eq:SP}, and let $\sgx$ be as in \eqref{eq:sgxdef}. Then the map $\Phi: \gt(\s_P)\to \gt(\sgx)$ defined by
	\begin{equation*}
	\Phi[s,x] = [\vp(s), x],\hspace{1cm} s\in \s_P, x\in D_{s^*s},
	\end{equation*}
	is an isomorphism of topological groupoids. In particular,
	\[
	\qpp \cong \ogx.
	\]
\end{theo}
\begin{proof}
	If $[s,x] = [t,x]$ then there exists a prefix $\alpha$ of $x$ such that $sA_\alpha = tA_\alpha$ and so $\vp(s)\vp(A_\alpha) = \vp(t)\vp(A_\alpha)$. Since $\vp(A_\alpha) = (\alpha,1_G, \alpha)$ is an idempotent and $x$ is in its domain, $[\vp(s),x] = [\vp(t), x]$ so that $\Phi$ is well-defined.
	
	To prove that $\Phi$ is injective, suppose that we have $\alpha, \gamma, \delta, \sigma\in X^*$, $x\in X^\NN$ and $g, k\in G$ such that $\Phi[[(\alpha, g)][(\gamma, 1_G)]^*B_n, x] =  \Phi[[(\delta, k)][(\sigma, 1_G)]^*B_m, x]$, i.e. $[(\alpha, g, \gamma), x] = [(\delta, k, \sigma), x]$. Then there exists a prefix $\mu$ of $x$ with $\mu = \gamma\gamma_1= \sigma\sigma_1$ and $(\alpha, g, \gamma)(\mu, 1_G, \mu) = (\delta, k, \sigma)(\mu, 1_G, \mu)$. Multiplying this out gives $(\alpha(g\cdot\gamma_1), \left.g\right|_{\gamma_1}, \mu) = (\delta(k\cdot\sigma_1), \left.k\right|_{\sigma_1}, \mu)$. But then the same calculation shows that $$[(\alpha, g)][(\gamma, 1_G)]^*B_nA_\mu B_{\max\{m,n\}} = [(\delta, k)][(\sigma, 1_G)]^*B_mA_\mu B_{\max\{m,n\}},$$ and since $\mu$ is a prefix of $x$ we have $$[[(\alpha, g)][(\gamma, 1_G)]^*B_n, x] = [[(\delta, k)][(\sigma, 1_G)]^*B_m, x]$$ so that $\Phi$ is injective. Surjectivity is clear, and that $\Phi$ is a groupoid homomorphism follows directly from Lemma~\ref{lem:semigrouphomomorphism}. Sets of the form $\Omega(s, C(\alpha))$ generate the topology on $\gt(\s_P)$ and it is straightforward to see that $\Phi$ maps $\Omega(s, C(\alpha))$ bijectively onto $\Omega(\vp(s), C(\alpha))$. Since sets of this type generate the topology on $\gt(\sgx)$, $\Phi$ is a homeomorphism.
	
	Since $\qpp$ is by definition equal to $C^*(\gt(\s_P))$ and we have that $\ogx$ is isomorphic to $C^*(\gt(\sgx))$ by \cite[Corollary~6.4]{EP17}, we have the second statement.
\end{proof}

\subsection*{Acknowledgement}
We thank the referee for thoroughly reading this paper and for the helpful suggestions and comments.







{\small 
	\textsc{Carleton University, School of Mathematics and Statistics. 4302 Herzberg Laboratories} \texttt{cstar@math.carleton.ca} 
}

\begin{thebibliography}{10}
	
	\bibitem{ABLS19}
	{\sc Afsar, Z., Brownlowe, N., Larsen, N.~S., and Stammeier, N.}
	\newblock Equilibrium states on right {LCM} semigroup {$C^*$}-algebras.
	\newblock {\em Int. Math. Res. Not. IMRN}, 6 (2019), 1642--1698.
	
	\bibitem{BOS18}
	{\sc Barlak, S., Omland, T., and Stammeier, N.}
	\newblock On the {$K$}-theory of {$C^{\ast}$}-algebras arising from integral
	dynamics.
	\newblock {\em Ergodic Theory Dynam. Systems 38}, 3 (2018), 832--862.
	
	\bibitem{Bl85}
	{\sc Blackadar, B.}
	\newblock Shape theory for {$C^\ast$}-algebras.
	\newblock {\em Math. Scand. 56}, 2 (1985), 249--275.
	
	\bibitem{BLS16}
	{\sc {Brownlowe}, N., {Larsen}, N.~S., and {Stammeier}, N.}
	\newblock {On C*-algebras associated to right LCM semigroups}.
	\newblock {\em Trans. Amer. Math. Soc. 369\/} (2017), 31--68.
	
	\bibitem{BLS18}
	{\sc Brownlowe, N., Larsen, N.~S., and Stammeier, N.}
	\newblock {$C^*$}-algebras of algebraic dynamical systems and right {LCM}
	semigroups.
	\newblock {\em Indiana Univ. Math. J. 67}, 6 (2018), 2453--2486.
	
	\bibitem{BRRW14}
	{\sc Brownlowe, N., Ramagge, J., Robertson, D., and Whittaker, M.~F.}
	\newblock Zappa--{S}z\'ep products of semigroups and their {C}*-algebras.
	\newblock {\em Journal of Functional Analysis 266}, 6 (2014), 3937 -- 3967.
	
	\bibitem{BS16}
	{\sc Brownlowe, N., and Stammeier, N.}
	\newblock The boundary quotient for algebraic dynamical systems.
	\newblock {\em J. Math. Anal. Appl. 438}, 2 (2016), 772--789.
	
	\bibitem{CaR16}
	{\sc Clark, L.~O., an~Heuf, A., and Raeburn, I.}
	\newblock Phase transitions of the {T}oeplitz algebras of {B}aumslag-{S}olitar
	semigroups.
	\newblock {\em Indiana Univ. Math. J. 65\/} (2016), 2137--2173.
	
	\bibitem{CEL15}
	{\sc Cuntz, J., Echterhoff, S., and Li, X.}
	\newblock On the {K}-theory of the {C}*-algebra generated by the left regular
	representation of an {O}re semigroup.
	\newblock {\em Journal of the European Mathematical Society 017}, 3 (2015),
	645--687.
	
	\bibitem{DW17}
	{\sc Dehornoy, P., and Wehrung, F.}
	\newblock Multifraction reduction {III}: the case of interval monoids.
	\newblock {\em J. Comb. Algebra 1}, 4 (2017), 341--370.
	
	\bibitem{DM14}
	{\sc Donsig, A.~P., and Milan, D.}
	\newblock Joins and covers in inverse semigroups and tight {C}*-algebras.
	\newblock {\em Bulletin of the Australian Mathematical Society 90\/} (8 2014),
	121--133.
	
	\bibitem{Ex08}
	{\sc Exel, R.}
	\newblock Inverse semigroups and combinatorial {$C^\ast$}-algebras.
	\newblock {\em Bull. Braz. Math. Soc. (N.S.) 39}, 2 (2008), 191--313.
	
	\bibitem{Ex09}
	{\sc Exel, R.}
	\newblock Tight representations of semilattices and inverse semigroups.
	\newblock {\em Semigroup Forum 79}, 1 (2009), 159--182.
	
	\bibitem{ExBook}
	{\sc Exel, R.}
	\newblock {\em Partial Dynamical Systems, Fell Bundles and Applications},
	vol.~224 of {\em Mathematical Surveys and Monographs}.
	\newblock American Mathematical Society, Providence, RI, 2017.
	
	\bibitem{Ex19}
	{\sc {Exel}, R.}
	\newblock {Tight and cover-to-join representations of semilattices and inverse
		semigroups}.
	\newblock {\em arXiv:1903.02911\/} (2019).
	
	\bibitem{EP17}
	{\sc Exel, R., and Pardo, E.}
	\newblock Self-similar graphs, a unified treatment of {K}atsura and
	{N}ekrashevych {C}*-algebras.
	\newblock {\em Advances in Mathematics 306\/} (2017), 1046 -- 1129.
	
	\bibitem{ES16}
	{\sc Exel, R., and Starling, C.}
	\newblock Self-similar graph {$C^*$}-algebras and partial crossed products.
	\newblock {\em J. Operator Theory 75}, 2 (2016), 299--317.
	
	\bibitem{HR90}
	{\sc Hancock, R., and Raeburn, I.}
	\newblock The {$C^*$}-algebras of some inverse semigroups.
	\newblock {\em Bull. Austral. Math. Soc. 42}, 2 (1990), 335--348.
	
	\bibitem{LL20}
	{\sc Laca, M., and Li, B.}
	\newblock Amenability and functoriality of right-{LCM} semigroup {C}*-algebras.
	\newblock {\em Proc. Amer. Math. Soc. 148}, 12 (2020), 5209--5224.
	
	\bibitem{LL21}
	{\sc Laca, M., and Li, B.}
	\newblock Dilation theory for right {LCM} semigroup dynamical systems.
	\newblock {\em arXiv:2102.08439 [math.OA]\/} (2021).
	
	\bibitem{LRRW14}
	{\sc Laca, M., Raeburn, I., Ramagge, J., and Whittaker, M.~F.}
	\newblock Equilibrium states on the cuntz–pimsner algebras of self-similar
	actions.
	\newblock {\em J. Funct. Anal. 266}, 11 (2014), 6619 -- 6661.
	
	\bibitem{LW15}
	{\sc Lawson, M.~V., and Wallis, A.~R.}
	\newblock A correspondence between a class of monoids and self-similar group
	actions {II}.
	\newblock {\em Internat. J. Algebra Comput. 25}, 4 (2015), 633--668.
	
	\bibitem{LiB19}
	{\sc Li, B.}
	\newblock Regular dilation and {N}ica-covariant representation on right {LCM}
	semigroups.
	\newblock {\em Integral Equations Operator Theory 91}, 4 (2019), Paper No. 36,
	35.
	
	\bibitem{Li12}
	{\sc Li, X.}
	\newblock Semigroup {C}*-algebras and amenability of semigroups.
	\newblock {\em Journal of Functional Analysis 262\/} (2012), 4302 -- 4340.
	
	\bibitem{Li13}
	{\sc Li, X.}
	\newblock Nuclearity of semigroup {C}*-algebras and the connection to
	amenability.
	\newblock {\em Advances in Mathematics 244}, 0 (2013), 626 -- 662.
	
	\bibitem{Li17}
	{\sc Li, X.}
	\newblock Partial transformation groupoids attached to graphs and semigroups.
	\newblock {\em Int. Math. Res. Not. IMRN}, 17 (2017), 5233--5259.
	
	\bibitem{MS14}
	{\sc {Milan}, D., and {Steinberg}, B.}
	\newblock {On inverse semigroup {C}*-algebras and crossed products}.
	\newblock {\em Groups, Geometry, and Dynamics 8}, 2 (2014), 485--512.
	
	\bibitem{Nek04}
	{\sc Nekrashevych, V.}
	\newblock Cuntz-{P}imsner algebras of group actions.
	\newblock {\em J. Operator Theory 52\/} (2004), 223--249.
	
	\bibitem{Nek05}
	{\sc Nekrashevych, V.}
	\newblock {\em Self-similar groups}, vol.~117 of {\em Mathematical Surveys and
		Monographs}.
	\newblock American Mathematical Society, Providence, RI, 2005.
	
	\bibitem{Nek09}
	{\sc Nekrashevych, V.}
	\newblock C*-algebras and self-similar groups.
	\newblock {\em J. reine angew. Math 630\/} (2009), 59--123.
	
	\bibitem{NS19}
	{\sc Neshveyev, S., and Stammeier, N.}
	\newblock The groupoid approach to equilibrium states on right {LCM} semigroup
	{C}*-algebras.
	\newblock {\em arXiv:1912.03141 [math.OA]\/} (2019).
	
	\bibitem{Ni92}
	{\sc Nica, A.}
	\newblock C*-algebras generated by isometries and {W}iener-{H}opf operators.
	\newblock {\em J. Operator Theory 27\/} (1992), 17--52.
	
	\bibitem{No14}
	{\sc Norling, M.~D.}
	\newblock Inverse semigroup {C}*-algebras associated with left cancellative
	semigroups.
	\newblock {\em Proceedings of the Edinburgh Mathematical Society (Series 2)
		57\/} (6 2014), 533--564.
	
	\bibitem{Pa99}
	{\sc Paterson, A.}
	\newblock {\em Groupoids, inverse semigroups, and their operator algebras}.
	\newblock Birkh\"{a}user, 1999.
	
	\bibitem{R80}
	{\sc Renault, J.}
	\newblock {\em A groupoid approach to {$C^{\ast} $}-algebras}, vol.~793 of {\em
		Lecture Notes in Mathematics}.
	\newblock Springer, Berlin, 1980.
	
	\bibitem{Sim20}
	{\sc Sims, A.}
	\newblock Hausdorff \'etale groupoids and their {C}*-algebras.
	\newblock In {\em Operator algebras and dynamics: groupoids, crossed products
		and {R}okhlin dimension}, F.~Perera, Ed. Birkh{\"a}user, 2020, ch.~7--11,
	pp.~63--120.
	
	\bibitem{Stam15}
	{\sc Stammeier, N.}
	\newblock On {${\rm C}^*$}-algebras of irreversible algebraic dynamical
	systems.
	\newblock {\em J. Funct. Anal. 269}, 4 (2015), 1136--1179.
	
	\bibitem{Stam17}
	{\sc Stammeier, N.}
	\newblock A boundary quotient diagram for right {LCM} semigroups.
	\newblock {\em Semigroup Forum 95}, 3 (2017), 539--554.
	
	\bibitem{StLCM}
	{\sc Starling, C.}
	\newblock Boundary quotients of {C}*-algebras of right {LCM} semigroups.
	\newblock {\em J. Funct. Anal. 268}, 11 (2015), 3326 -- 3356.
	
	\bibitem{Tol17}
	{\sc Tolich, I.}
	\newblock {\em C*-algebras generated by semigroups of partial isometries}.
	\newblock PhD thesis, University of Otago, Otago, New Zealand, 2017.
	
\end{thebibliography}
\end{document}